\documentclass[11pt]{article}
\usepackage[a4paper,centering,vscale=0.75,hscale=0.7]{geometry}
\usepackage{mathrsfs,amssymb,amsmath,amsthm}

\newcommand{\enne}{\mathbb{N}}
\newcommand{\erre}{\mathbb{R}}
\newcommand{\B}{\mathscr{B}}
\newcommand{\E}{\mathbb{E}}
\newcommand{\LL}{\mathscr{L}}
\renewcommand{\P}{\mathbb{P}}

\DeclareMathOperator{\dom}{dom}
\newcommand{\ip}[2]{\langle #1,#2 \rangle}
\newcommand{\bip}[2]{\left\langle #1,#2 \right\rangle}

\newtheorem{thm}{Theorem}[section]
\newtheorem{lemma}[thm]{Lemma}
\newtheorem{prop}[thm]{Proposition}
\newtheorem{coroll}[thm]{Corollary}
\theoremstyle{definition}
\newtheorem{defi}[thm]{Definition}
\newtheorem{ex}[thm]{Example}
\theoremstyle{remark}
\newtheorem{rmk}[thm]{Remark}

\numberwithin{equation}{section}

  \renewenvironment{thebibliography}[1]{%
    \begin{oldthebibliography}{#1}%
      \setlength{\parskip}{0ex}%
      \setlength{\itemsep}{1ex}%
  }%
  {%
    \end{oldthebibliography}%
  }

\title{On smoothing properties of transition semigroups associated to
  a class of SDEs with jumps}
\author{Seiichiro Kusuoka%
  \thanks{Gradute School of Science, Kyoto University.}  \and Carlo
  Marinelli%
  \thanks{Department of Mathematics, University College London, Gower
    Street, London WC1E 6BT, United Kingdom, and Facolt\`a di
    Economia, Universit\`a di Bolzano, Italy.}}
\date{14 August 2012}

\begin{document}
\maketitle

\begin{abstract}
  We prove smoothing properties of nonlocal transition semigroups
  associated to a class of stochastic differential equations (SDE) in
  $\erre^d$ driven by additive pure-jump L\'evy noise. In particular,
  we assume that the L\'evy process $Z$ driving the SDE is the sum of
  a subordinated Wiener process $Y$ (i.e. $Y=W \circ T$, where $T$
  is an increasing pure-jump L\'evy process starting at zero and
  independent of the Wiener process $W$) and of an arbitrary L\'evy
  process $\xi$ independent of $Y$, that the drift coefficient is
  continuous (but not necessarily Lipschitz continuous) and grows not
  faster than a polynomial, and that the SDE admits a Feller weak
  solution. By a combination of probabilistic and analytic methods, we
  provide sufficient conditions for the Markovian semigroup associated
  to the SDE to be strong Feller and to map $L_p(\erre^d)$ to
  continuous bounded functions. A key intermediate step is the study
  of regularizing properties of the transition semigroup associated to
  $Y$ in terms of negative moments of the subordinator $T$.
\end{abstract}

\section{Introduction}
The purpose of this work is to prove smoothing properties for the
(Markovian) semigroup generated by the (weak) solution to a stochastic
differential equation in $\erre^d$ of the type
\begin{equation}     \label{eq:sde}
dX_t = b(X_t)\,dt + dZ_t, \qquad X_0=x,
\end{equation}
where $Z$ is a pure-jump L\'evy process which can be written as
$Z=Y+\xi$, where $Y$ is obtained by subordination of a Wiener process
$W$ with non-degenerate covariance matrix, and $\xi$ is a further
L\'evy process independent of $Y$, on which
no further assumption is imposed.
In particular, assume that \eqref{eq:sde} admits a Markovian weak
solution denoted by $(X^x_t)_{t\geq 0}$, and that the semigroup
$(P^X_t)_{t>0}$, $P^X_tf(x) :=\E f(X^x_t)$, for $f$ Borel measurable
and bounded, is Feller, i.e. that $P_t^X$ leaves invariant the space
of bounded continuous functions. We look for sufficient conditions on
the L\'evy process $Y$ and on the drift coefficient $b$ such that
$P^X_t$ is strong Feller, resp. $L_p$-strong Feller, for all $t>0$,
i.e. that $P^X_t$ maps bounded Borel measurable functions,
resp. $L_p(\erre^d)$, to bounded continuous functions.

We proceed in two steps: first we study the regularizing properties of
the semigroups $(P^Y_t)_{t>0}$ and $(P^Z_t)_{t>0}$ associated,
respectively, to the L\'evy process $Y$ and $Z$ (an issue which is
interesting in its own right); then we show that the semigroup
associated to $X$ inherits, at least in part, the regularizing
properties of $P^Y$. In particular, in the former step we provide
conditions in terms of the existence of negative moments of the
subordinator $T_t$ such that $P^Y_t$ (hence also $P^Z_t$, as we
shall see) maps $\B_b(\erre^d)$, the space of bounded Borel measurable
functions on $\erre^d$, or $L_p(\erre^d)$, to $C^k_b(\erre^d)$, $k \in
\enne$. The latter step is a perturbation argument relying on
Duhamel's formula.

The strong Feller property for semigroups generated by solutions to
stochastic (both ordinary and partial) differential equations with
jumps is usually obtained by suitable versions of the
Bismut-Elworthy-Li formula (see e.g. \cite{cm:JFA10,PriZa}). However,
this method requires the driving noise to have a non-degenerate
diffusive component, therefore it is not applicable to our problem.
For some special classes of equations driven by pure jump noise other
approaches have been devised: for instance, in \cite{PriZa:stru} the
authors prove the strong Feller property for the semigroup generated
by the solution to a semilinear SPDE driven by an infinite sum of
one-dimensional independent stable processes, assuming that the
nonlinearity in the drift term is Lipschitz continuous and
bounded. Their proofs rely on finite-dimensional projections and
specific properties of stable measures.

Let us also mention that the problem we are dealing with admits a
clear analytic interpretation. In fact, an application of It\^o's
formula yields that the generator $L$ of $P^X$ acts on smooth
functions as follows:
\[
L\phi(x) = \ip{b(x)}{\nabla \phi(x)} 
+ \int_{\erre^d \setminus \{ 0\}} \bigl(
\phi(x+y) - \phi(x) - \ip{\nabla \phi(x)}{y}1_{\{|y|<1\}} \bigr) \,m^Z(dy),
\]
where $m^Z$ stands for the L\'evy measure of $Z$. 
Therefore, being somewhat formal, our problem is equivalent to
establishing regularity (specifically, continuity and boundedness) of
the solution at time $t>0$ to the non-local parabolic Kolmogorov
equation $\partial_t u = Lu$, $u(0)=u_0$, where the initial datum
$u_0$ is taken either Borel measurable and bounded, or belonging to an
$L_p$ space, on $\erre^d$. Using analytic methods, related problems
have already been investigated, e.g. in \cite{MikuPrag:HS}, where it
is assumed, roughly speaking, that $Z$ is a perturbation of an
$\alpha$-stable process. For more recent results, covering also
nonlinear equations, one could see e.g. \cite{CaChaVas} and references
therein. It does not seem, however, that our results can be recovered
by available regularity estimates for non-local parabolic equations.

One should also recall that there exists a rich literature on
existence and regularity of densities for solutions to SDEs with
jumps, mostly applying suitable versions of Malliavin calculus (see
e.g. \cite{BalCle,IshiKu,Kun:*,Sei2,Lea:sub} and references
therein). Such existence and regularity result may be used, in turn,
to prove that a Feller process is strong Feller (see
\cite[Corollary~2.2]{SchiWa} for a general result in this
direction).
In general, however, applying Malliavin calculus directly to an SDE
driven by a L\'evy process usually requires that the L\'evy measure of
the driving noise admits sufficiently many finite moments (see
e.g. \cite{IshiKu,Kun:*}). This problem of course does not occur in
the case of SDEs driven by Brownian motion
(cf.~\cite{KuStr2}). Moreover, the coefficients of the SDE are assumed
to be sufficiently smooth (usually of class $C^2_b$ at least), and it
is difficult to weaken this hypothesis very much.
The smoothing properties proved in this paper, on the other hand, are
applicable also to equations driven by L\'evy processes whose L\'evy
measure possesses only moments of very low order (such as stable
processes), and with a drift coefficient that is not Lipschitz
continuous.

From the analytic point of view, speaking again somewhat formally, the
distribution $\mu_t$ of the solution $X_t^x$, $t \geq 0$, to
\eqref{eq:sde} solves, in the sense of distributions, the non-local
parabolic equation for probability measures $\partial_t \mu = L^*\mu$,
with initial datum equal to a Dirac measure centered at $x$, where
$L^*$ stands for the formal adjoint of $L$.  Note that, in our
specific situation, assuming for the sake of simplicity that the
generator of $Z$ is symmetric (which is certainly the case if $\xi
\equiv 0)$, one can write
\[
L^*\phi(x) = \operatorname{div}(b\phi) + \int_{\erre^d \setminus \{ 0\}} \bigl(
\phi(x+y) - \phi(x) - \ip{\nabla \phi(x)}{y}1_{\{|y|<1\}} \bigr)
\,m^Z(dy).
\]
Unfortunately, however, we are not aware of any existence and
regularity results for non-local Fokker-Planck equations of the type
$\partial_t\mu=L^*\mu$. Nonetheless, it is interesting to note that,
if one knows a priori (or assumes, as we do) that $X$ has the Feller
property, such results would imply regularity properties of the
solution to the Kolmogorov equation $\partial_t u=Lu$.

Let us also recall that subordination has already been used to
establish the strong Feller property for some classes of Markov
processes with jumps (cf.~\cite{GoRoWa,Sei2,Lea:sub}). We would like
to stress, however, that it seems difficult to deduce
properties of semigroups generated by operators such as $L$ from the
properties of semigroups generated by corresponding local operators of
the type $L_\ell \phi = \ip{b}{\nabla\phi}+\Delta\phi$. Using
more probabilistic language, it is not clear at all whether one can
establish properties of the solution to an SDE of the type
\eqref{eq:sde} (assuming $\xi \equiv 0$ for simplicity) studying the
process obtained by subordination with $T$ of the solution to the same
SDE with $Y$ replaced by the Wiener process $W$. These considerations
and the need to treat semigroups generated by non-local operators with
drift are the main motivations for our approach.

Smoothing properties of equations with multiplicative noise (i.e. with a ``diffusion'' coefficient depending on $X$ in front of the noise in \eqref{eq:sde}) are also an interesting problem, but unfortunately it seems difficult to adapt our techniques to this case. On the other hand, if both the drift and the diffusion coefficients are sufficiently smooth (i.e. at least of class $C_b^2$), and the noise $Z$ is $\alpha$-stable, we show that one can apply Malliavin calculus methods to prove that the solution generates a strong Feller semigroup.

\medskip

The paper is organized as follows: we collect in Section
\ref{sec:prel} some basic preliminaries, and, in Section
\ref{sec:smoothY}, we extensively study regularizing properties of
semigroups associated to subordinate Wiener processes. In particular,
we derive estimates on the $k$-th order Fr\'echet derivative of such
semigroups in terms of negative moments of the corresponding
subordinators. These estimates are an essential ingredient for the
proof of the main results in Section \ref{sec:main}. Finally, in
Section \ref{sec:Malliavin} we consider the case of equations with
multiplicative stable noise: under smoothness
assumptions on the coefficients, we prove the strong Feller property
of the transition semigroup, applying some results that were obtained
in \cite{Sei2} by a suitable version of Malliavin calculus.

\subsection*{Acknowledgments}
Part of the work for this paper was carried out when both authors
where visiting the Institute for Applied Mathematics of the University
of Bonn. We are particularly grateful to our hosts Sergio Albeverio
and Andreas Eberle. The first-named author also gratefully acknowledges the support of the Japan Society for the Promotion of Science through the Excellent Young Researcher Overseas Visit Program.

\section{Preliminaries}     \label{sec:prel}
\subsection{Notation and terminology}
We shall denote the set of bounded Borel measurable functions on
$\erre^d$ by $\B_b(\erre^d)$. Note that $\B_b(\erre^d)$, endowed with
the norm $\|\phi\|_\infty:=\sup_{x\in \erre^d} |\phi(x)|_{\erre^d}$,
is a Banach space. The subset of $\B_b(\erre^d)$ consisting of functions
with compact support is denoted by $\B_{b,c}(\erre^d)$. The space of bounded
continuous functions on $\erre^d$ will be denoted by $C_b(\erre^d)$, and,
similarly, $C^k_b(\erre^d)$, $k \in \mathbb{N}$, will denote the space of
continuously differentiable functions with bounded derivatives up to
order $k$. The space of infinitely differentiable functions with
compact support is denoted by $C^\infty_c(\erre^d)$.  Given a function
$f:\erre^d \to \erre$ and a multiindex $\alpha \in \enne_0^d$ (where
$\enne_0:=\enne \cup \{0\}$), we shall use the standard notation
\[
\partial^\alpha f(x_0) := \frac{\partial^{\alpha_1}}{\partial x_1^{\alpha_1}}
\frac{\partial^{\alpha_2}}{\partial x_2^{\alpha_2}} \cdots
\frac{\partial^{\alpha_d}}{\partial x_d^{\alpha_d}} f(x_0),
\qquad x_0 \in \erre^d,
\]
and $|\alpha|=\alpha_1+\alpha_2+\cdots+\alpha_d$.
Given $n \in \enne$, the $n$-th Fr\'echet derivative of $f$ at a point
$x_0 \in \erre^d$ will be denoted by $D^nf(x_0)$. Recall that
$D^nf(x_0)$ can be identified with an element of $\LL_n(\erre^d)$, the
space of $n$-multilinear mappings on $\erre^d$.

Lebesgue spaces are denoted by $L_p(\erre^d)$, $1 \leq p \leq \infty$, and
the corresponding Sobolev spaces by $W_p^m(\erre^d)$, $m \in \enne$. In the
following we shall sometimes denote function spaces without mentioning
the underlying space $\erre^d$.  An expression of the type $E
\hookrightarrow F$ means that the space $E$ is continuously embedded
into the space $F$.
If $a \leq N b$ for some positive constant $N$ we shall often write $a
\lesssim b$.

\medskip

We recall standard terminology, plus some slightly non-standard one
needed for the purposes of this work.  Let us recall that a linear
positivity preserving operator $A:\B_b(\erre^d) \to \B_b(\erre^d)$ is called
sub-Markovian if it is contracting, i.e. if $\|A\phi\|_\infty \leq
\|\phi\|_\infty$ for all $\phi\in\B_b(\erre^d)$.
\begin{defi}
  A sub-Markovian operator $A$ on $\B_b(\erre^d)$ is called:
  \begin{itemize}
  \item[(i)] \emph{Feller} if $A\bigl(C_b(\erre^d)\bigr) \subseteq
    C_b(\erre^d)$;
  \item[(ii)] \emph{strong Feller} if $A\bigl(\B_b(\erre^d)\bigr)
    \subseteq C_b(\erre^d)$;
  \item[(iii)] $c$-\emph{strong Feller} if
    $A\bigl(\B_{b,c}(\erre^d)\bigr) \subseteq C_b(\erre^d)$;
  \item[(iv)] $k$-\emph{smoothing}, $k \in \mathbb{N}$, if
    $A\bigl(\B_b(\erre^d)\bigr) \subseteq C^k_b(\erre^d)$.
  \end{itemize}
\end{defi}
\begin{defi}
  Given $1 \leq p \leq \infty$, a linear bounded operator $A$ from
  $L_p(\erre^d)$ to $C_b(\erre^d)$ will be called $L_p$-strong
  Feller.
\end{defi}
\begin{rmk}
  Note that, in general, a sub-Markovian operator on $\B_b(\erre ^d)$
  may not even be defined on any space $L_p(\erre^d)$, $1 \leq p \leq
  \infty$. Conversely, a map from $L_p(\erre^d)$ to $C_b(\erre^d)$ may
  not be defined on $\B_b(\erre^d)$. Therefore, in general, an
  $L_p$-strong Feller operator may not be Feller, and viceversa.
  However, by Lemma \ref{lm:Feller} below, if an operator is Feller
  and $L_p$-strong Feller, then it is strong Feller.
  A necessary condition for the last definition to make sense is that
  the operator $A$ maps indicator functions of sets of Lebesgue
  measure zero to zero (i.e. to the continuous function equal to
  zero). This condition is clearly not satisfied by all sub-Markovian
  operators on $\B_b(\erre^d)$, but it is indeed satisfied if $A$ is
  of the type
  \[
  Af = \int_{\erre^d} k(\cdot,y)f(y)\,dy \qquad \forall f \in
  \B_b(\erre^d),
  \]
  with appropriate measurability conditions on $k$.
\end{rmk}
As it is customary, one says that a Markov process is Feller (or
strong Feller, etc.) to mean that its transition semigroup is made of
Feller operators.

\smallskip

In the following lemma we provide a simple yet useful criterion
to establish that a Feller operator is strong Feller.
\begin{lemma}     \label{lm:Feller}
  Let $A$ be a Feller operator on $\B_b(\erre^d)$. Then $A$ is strong
  Feller if and only if it is $c$-strong Feller.
\end{lemma}
\begin{proof}
  We only have to prove that the $c$-strong Feller property implies
  the strong Feller property. Let $f \in \B_b$, and
  $\{\chi_k\}_{k\in\enne} \subset C^\infty_c$ a sequence of cutoff
  functions such that $0 \leq \chi_k \leq 1$ for all $k$ and $\chi_k
  \uparrow 1$ as $k \to \infty$. Since $A$ is positivity preserving,
  we have $A\chi_k \uparrow A1$ as $k \to \infty$, and $A\chi_k \in
  C_b$ for all $k$ because (obviously) $\chi_k \in \B_{b,c}$.  If an
  increasing sequence of continuous functions converges pointwise to a
  continuous function, the convergence is locally uniform by Dini's
  theorem.  Therefore, we infer that $A\chi_k \to A1$ locally
  uniformly as $k \to \infty$. Using again that $A$ is sub-Markovian
  and $\chi_k \leq 1$ for all $k$, we have
  \[
  \big| Af - A(\chi_k f) \big| = \big| A((1-\chi_k)f) \big|
  \leq
  \|f\|_\infty \big| A1 - A\chi_k \big|,
  \]
  which implies that $A(\chi_k f) \to Af$ locally uniformly. But
  $\chi_kf \in \B_{b,c}$, hence $A(\chi_k f) \in C_b$, and we can
  conclude that $Af$ is continuous as it is the local uniform limit of
  a sequence of continuous functions. That $Af$ is bounded is obvious
  by sub-Markovianity of $A$.
\end{proof}
By inspection of the proof, one realizes that one could also assume
that $A$ is Markovian (i.e. sub-Markovian and
conservative\footnote{The operator $A$ is conservative if $A1=1$.}),
rather than Feller. The previous lemma then has the following
immediate consequence, which can indeed be quite useful.
\begin{coroll}
  Let $A$ be a Markovian operator on $\B_b(\erre^d)$. Then $A$ is
  strong Feller if and only if it is $c$-strong Feller.
\end{coroll}

\subsection{Subordinators}
By subordinator we shall always understand an increasing L\'evy process $T:\erre_+ \to \erre_+$ such that $T_0=0$ and $T_t>0$ for $t>0$. Then
one has, for $\lambda \geq 0$,
\[
\E e^{-\lambda T_t} = e^{-t \Phi(\lambda)},
\]
where $\Phi: [0,\infty[ \to [0,\infty[$ is such that
\[
\Phi(\lambda) = \int_{]0,\infty[} (1-e^{-\lambda x})\,m(dx).
\]
Here $m$ is the L\'evy measure of $T$, whose support is contained in
$[0,\infty[$, and satisfies
\[
\int_{]0,\infty[} (1 \wedge x)\,m(dx) < \infty.
\]
For a proof of the above facts (and much more) one can consult e.g.
\cite{Bertoin}.

\subsection{Function spaces}
We recall some definitions and results on H\"older and Bessel
potential spaces, referring to \cite{Tri1} for a complete treatment as
well as for all unexplained notation. Bessel spaces are only used in
\S\ref{ssec:Bess}.

Given a real non-integer number $s>0$, let us set $s=[s]+\{s\}$, with
$[s] \in \enne$ and $0 < \{s\} < 1$. The H\"older space $C^s_b(\erre^d)$
is defined as the set of functions $f \in C_b^{[s]}(\erre^d)$ such that
\[
\|f\|_{C^s_b(\erre^d)} := \sum_{|\alpha| \leq [s]}
\bigl\| \partial^\alpha f \bigr\|_{L_\infty(\erre^d)} +
\sum_{|\beta|=[s]} \sup_{x \neq y}
\frac{\bigl|\partial^\beta f(x)-\partial^\beta f(y)\bigr|}{|x-y|^{\{s\}}}
< \infty.
\]
The Zygmund space on $\erre^d$ of order $s\in\erre$,
cf.~\cite[p.~36]{Tri1}, will be denoted by
$\mathscr{C}^s(\erre^d)$. Recall that one has $\mathscr{C}^s(\erre^d)
= C^s(\erre^d)$ for all real non-integer $s>0$ (see \cite[Rmk.~3,
p.~38]{Tri1}).

The Bessel potential space $H_p^s(\erre^d)$, with $1 < p < \infty$ and
$s \in \erre$, is the space of Schwartz distributions $f \in
\mathscr{S}'(\erre^d)$ such that $(I-\Delta)^{s/2} f \in
L_p(\erre^d)$, with
\[
\|f\|_{H_p^s(\erre^d)} = \|(I-\Delta)^{s/2} f\|_{L_p(\erre^d)}.
\]
For convenience, let us also define the homogeneous norm
\[
\|f\|_{\dot{H}_p^s(\erre^d)} := \|(-\Delta)^{s/2} f\|_{L_p(\erre^d)}.
\]
The one has
\begin{equation}     \label{eq:homo}
\|f\|_{H_p^s(\erre^d)} \eqsim \|f\|_{L_p(\erre^d)} + \|f\|_{\dot{H}_p^s(\erre^d)}.
\end{equation}
As is well known, if $m \in \enne$ one has
$W_p^m(\erre^d)=H_p^m(\erre^d)$.

The following embedding result for Bessel potential spaces is
certainly known, but we have not been able to find it anywhere in this
formulation. We include a proof (admittedly very cryptic, but with
precise references) for completeness.
\begin{lemma}[Sobolev embedding theorem]     \label{lm:Bess}
  Let $1<p<\infty$ and $s>d/p$. Then $H_p^s(\erre^d) \hookrightarrow
  \mathscr{C}^{s-d/p}(\erre^d)$. In particular, if $s-d/p \not\in
  \enne$, then $H_p^s(\erre^d) \hookrightarrow
  C_b^{s-d/p}(\erre^d)$.
\end{lemma}
\begin{proof}
  We have
  \[
  H_p^s(\erre^d) = F_{p,2}^s(\erre^d) \hookrightarrow
  F^s_{p,\infty}(\erre^d) \hookrightarrow B^s_{p,\infty}(\erre^d),
  \]
  where we have used \cite[Thm.~(i), p.~88]{Tri1} and \cite[Prop.~2,
  p.~47]{Tri1}, in this order. By \cite[Thm.~(i), p.~129]{Tri1} we
  also have
  \[
  B^s_{p,\infty}(\erre^d) \hookrightarrow
  B^{s_1}_{\infty,\infty}(\erre^d),
  \]
  where $s_1=s-d/p$. Since $s_1>0$, \cite[Coroll.~(i), p.~113]{Tri1}
  implies
  $\mathscr{C}^{s_1}(\erre^d)=B_{\infty,\infty}^{s_1}(\erre^d)$, hence
  also $H_p^s(\erre^d) \hookrightarrow \mathscr{C}^{s_1}(\erre^d)$.
  The proof is concluded recalling that, as already mentioned above,
  one has $\mathscr{C}^{s_1}(\erre^d) = C_b^{s_1}(\erre^d)$, if $s_1$
  is not integer.
\end{proof}

\section{Smoothing properties of subordinated Wiener
processes}     \label{sec:smoothY}
Let $(\Omega,\mathcal{F},(\mathcal{F}_t)_{t \geq 0},\P)$ be a filtered
probability space, on which all random variables and processes will be
defined. Let $W$ be a standard $\erre^d$-valued Wiener process
(i.e. with covariance operator equal to the identity) and $T$ be a
subordinator with infinite lifetime and independent from $W$. Let us
define the Markovian stochastic process $Y:=W \circ T$,
i.e. $Y_t:=W_{T_t}$ for all $t \geq 0$, and its associated semigroup
\[
P_t^Y f(x) := \E f(x+Y_t), \qquad f \in \mathscr{B}_b(\erre^d).
\]
The process $Y$ is often referred to as the Wiener process $W$
subordinated to $T$.

In the sequel we shall denote the density of the random variable
$W_t$, $t \in [0,\infty[$, by $p_t:\erre^d \to \erre$, with
\[
p_t(y) = \frac{1}{(2\pi t)^{d/2}} e^{-|y|^2/(2t)}.
\]
With a slight (but innocuous) abuse of terminology, we shall refer to
the function $p:(t,y) \mapsto p_t(y)$ as the transition density of $W$
(and similarly for other translation-invariant processes), or as the
heat kernel on $\erre^d$.

The following elementary lemma relates the transition density of the
L\'evy process $Y$ to the one of $W$ and shows that the strong Feller
property of $W$ is inherited by $Y$.
\begin{lemma}     \label{lm:dens}
  The process $Y=W\circ T$ admits a transition density $(t,y) \mapsto p^Y_t(y)$
  given by
  \[
  p_t^Y(y) = \int_0^\infty p_s(y)\,\nu_t(ds),
  \]
  where $\nu_t:= \P \circ T_t^{-1}$, $t>0$, stands for the law of the
  random variable $T_t$. In particular, $P_t^Y$ is strong Feller for
  all $t>0$.
\end{lemma}
\begin{proof}
  Let $f \in \B_b(\erre^d)$. Then, using properties of conditional
  expectation and recalling that $W$ and $T$ are independent, one has
  \[
  \E f(Y_t) = \int_0^\infty\!\!\int_{\erre^d}
  f(y) p_s(y)\,dy\,\nu_t(ds).
  \]
  The conclusion then follows by Fubini's theorem, since $p$ is
  positive and
  \[
  \int_0^\infty\!\!\int_{\erre^d}
  |f(y)| p_s(y)\,dy\,\nu_t(ds) \leq \|f\|_\infty.
  \]
  Another immediate application of Tonelli's theorem (or just
  recalling that $Y_t$ is finite $\P$-a.s. for all $t \geq 0$) shows
  that $p_t^Y \in L_1$ and $\|p_t^Y\|_1=1$.  In particular, $P_t^Yf =
  f \ast p^Y_t$, with $f \in L_\infty$ and $p^Y_t \in L_1$, hence
  $\|P_t^Yf\|_{L_\infty} \leq \|p^Y\|_{L_1} \|f\|_\infty$ by Young's
  inequality. We only have to show that $P_t^Yf$ is continuous: let
  $(\phi_n)_{n\in\enne}$ be a sequence of functions in
  $C^\infty_c(\erre^d)$ such that $\phi_n \to p_t^Y$ in
  $L_1(\erre^d)$. Then clearly $f \ast \phi_n \in C_b(\erre^d)$ for
  all $n \in \enne$, and
  \[
  \bigl\| f \ast p_t^Y - f \ast \phi_n \bigr\|_{L_\infty} \leq
  \|f\|_\infty \, \bigl\| p_t^Y - \phi_n \bigr\|_{L_1}
  \xrightarrow{n\to\infty} 0,
  \]
  which implies that $P_t^Yf \in C(\erre^d)$ as uniform limit of
  continuous functions.
\end{proof}
We are going to use some well-known properties of the heat kernel on
finite dimensional Euclidean spaces. In particular, observing that one
can write
\begin{equation}
\label{eq:hk}
p_t(x) = \frac{1}{(2\pi)^{d/2}} t^{-d/2} \phi\big(\frac{|x|^2}{t}\big),
\qquad
\phi(r)=e^{-r/2},
\end{equation}
it is immediately seen (and well known) that $x \mapsto p_t(x) \in
\mathscr{S}(\erre^d)$ for all $t>0$, where $\mathscr{S}(\erre^d)$
stands for the Schwartz space of smooth functions with rapid decrease
at infinity. 

Before we proceed, we need to recall some facts about Hermite
polynomials (see e.g. \cite[p.~7]{Bog-GM}, but note that we use a
different normalization).
For $n \in \enne_0$, the Hermite polynomial of degree $n$ is
\[
H_n(y) = (-1)^n e^{y^2/2} \frac{d^n}{dy^n}e^{-y^2/2},
\qquad y \in \erre
\]
Let us recall that, if $n$ is even, one has
\[
H_n(y)=\sum_{j=0}^{n/2} a_{n,n-2j}y^{n-2j},
\]
and, for $n$ odd,
\[
H_n(y)=\sum_{j=0}^{(n-1)/2} a_{n,n-2j}y^{n-2j},
\]
where $|a_{n,m}|$ is the number of unordered partitions of the set
$\{1,2,\ldots,n\}$ into $m$ singletons and $(n-m)/2$ unordered pairs.
Given a multiindex $\alpha \in \enne_0^d$ and $x=(x_1,\ldots,x_d) \in
\erre^d$, we set
\[
H_\alpha(x) := \prod_{k=1}^d H_{\alpha_k}(x_k).
\]

\medskip

Let us now give an expression for the general mixed partial
derivatives of $p_t$.
\begin{lemma}     \label{lm:mix}
  Let $x \in \erre^d$ and $t>0$. For any $\alpha \in \mathbb{N}_0^d$, one has
  \[
  \partial^\alpha p_t(x) = t^{-|\alpha|/2} (-1)^{|\alpha|}
  \, H_\alpha(t^{-1/2}x) \, p_t(x)
  \]
\end{lemma}
\begin{proof}
  Writing $x=(x_1,\ldots,x_d)$, one has
  \[
  p_t(x) = t^{-d/2} p_1\bigl( x/\sqrt{t} \bigr) = 
  t^{-d/2} \prod_{k=1}^d p_1\bigl( x_k/\sqrt{t} \bigr),
  \]
  therefore
  \[
  \partial^\alpha p_t(x) = t^{-d/2} \prod_{k=1}^d
  D_{x_k}^{\alpha_k} p_1\bigl( x_k/\sqrt{t} \bigr) =
  t^{-d/2} \prod_{k=1}^d t^{-\alpha_k/2}
  p_1^{(\alpha_k)}\bigl( x_k/\sqrt{t} \bigr).
  \]
  Recalling the definition of Hermite polynomials, one has
  \[
  p_1^{(n)}(y) = \frac{1}{\sqrt{2\pi}} e^{-y^2/2} (-1)^n 
  H_n(y) = (-1)^n H_n(y) \, p_1(y),\quad y\in \erre
  \]
  for all $n \in \enne$. This yields
  \begin{align*}
    \partial^\alpha p_t(x) &= t^{-d/2} t^{-|\alpha|/2} (-1)^{|\alpha|}
    \prod_{k=1}^d H_{\alpha_k}(t^{-1/2}x_k) p_1(t^{-1/2}x_k)\\
    &= t^{-|\alpha|/2} (-1)^{|\alpha|} \, H_\alpha(t^{-1/2}x) \, p_t(x).
    \qedhere
  \end{align*}
\end{proof}

\bigskip

With some more effort one can obtain an expression for the general
Fr\'echet derivative of $p_t$ of order $n \in \enne$. To this purpose,
given any $x \in \erre^d$, let us first associate to the Hermite
polynomial $H_n$ an $n$-linear operator $\tilde{H}_n(x) \in
\LL_n(\erre^d)$. It is sufficient to associate to any monomial of the
form $a_{n,m}x^m$, $0 \leq m \leq n$, the following operator in
$\LL_n(\erre^d)$:
\[
(h_1,\ldots,h_n) \mapsto
\sum_{\beta \in B(n,m)} \bip{x}{h_{\beta_1}} \bip{x}{h_{\beta_2}} \cdots
\bip{x}{h_{\beta_m}} \bip{h_{\beta_{m+1}}}{h_{\beta_{m+2}}} \cdots
\bip{h_{\beta{n-1}}}{h_{\beta_n}},
\]
where $B(n,m)$ is the set of all unordered partitions of the set
$\{1,2,\ldots,n\}$ into $m$ singletons and $(n-m)/2$ unordered pairs,
and we identify $\beta \in B(n,m)$ with the corresponding
rearrangement of the set $\{1,2,\ldots,n\}$.
Let us give an explicit example: given the Hermite polynomial
\[
H_4(y) = y^4 -6y^2 + 3,
\]
one has, for any $x\in\erre^d$, $\tilde{H}_4(x) = \tilde{H}_{41}(x) -
\tilde{H}_{42}(x) + \tilde{H}_{43}(x)$, where
\begin{gather*}
\tilde{H}_{41}(x): 
(h_1,\ldots,h_4) \mapsto \ip{x}{h_1} \cdots \ip{x}{h_4}\\
\begin{aligned}
\tilde{H}_{42}(x): 
(h_1,\ldots,h_4) &\mapsto \ip{x}{h_1} \ip{x}{h_2} \ip{h_3}{h_4}
+ \ip{x}{h_1} \ip{x}{h_3} \ip{h_2}{h_4}\\
&\qquad\qquad + \ip{x}{h_1} \ip{x}{h_4} \ip{h_2}{h_3}\\
&\qquad 
+ \ip{x}{h_2} \ip{x}{h_3} \ip{h_1}{h_4}
+ \ip{x}{h_2} \ip{x}{h_4} \ip{h_1}{h_3}\\
&\qquad 
+ \ip{x}{h_3} \ip{x}{h_4} \ip{h_1}{h_2},
\end{aligned}\\
\tilde{H}_{43}(x):
(h_1,\ldots,h_4) \mapsto \ip{h_1}{h_2} \ip{h_3}{h_4}
+ \ip{h_1}{h_3} \ip{h_2}{h_4}
+ \ip{h_1}{h_4} \ip{h_2}{h_3}.
\end{gather*}

With these preparations, we can state the following lemma.
\begin{lemma}     \label{lm:Hermite}
  Let $x \in \erre^d$ and $n \in \enne$. Then one has
  \begin{equation}     \label{eq:Fre-hk}
  D^n p_t(x) = (-1)^n t^{-n/2} p_t(x) \tilde{H}_n(t^{-1/2}x).
  \end{equation}
\end{lemma}
\begin{proof}
  Repeating the computations leading to the definition of the Hermite
  polynomials, replacing the usual derivative on the real line with
  the Fr\'echet derivate, one arrives at
  \[
  D^np_1(x) = (-1)^np_1(x)\tilde{H}_n(x).
  \]
  Recalling that $p_t(x) = t^{-d/2} p_1(t^{-1/2}x)$, hence
  \[
  D^np_t(x) = t^{-n/2} \, t^{-d/2} D^np_1(t^{-1/2}x),
  \]
  we are left with
  \begin{align*}
    D^np_t(x) &= (-1)^n t^{-n/2} \, t^{-d/2} p_1(t^{-1/2}x)
                 \tilde{H}_n(t^{-1/2}x)\\
    &= (-1)^n t^{-n/2} p_t(x) \tilde{H}_n(t^{-1/2}x).
    \qedhere
  \end{align*}
\end{proof}

It should be noted that, for lower values of $n$, an expression for
$D^np_t(x)$ can be easily obtained by (Fr\'echet) differentiation of
\eqref{eq:hk}. For instance,
\begin{align*}
Dp_t(x) &= - \frac{1}{(2\pi)^{d/2}} t^{-d/2-1}
\phi\big(\frac{|x|^2}{t}\big) \ip{x}{\cdot} =
-t^{-1} p_t(x) \ip{x}{\cdot},\\
D^2 p_t(x) &= \frac{1}{(2\pi)^{d/2}} t^{-d/2-2}
\phi\big(\frac{|x|^2}{t}\big) \ip{x}{\cdot}\,\ip{x}{\cdot}
- \frac{1}{(2\pi)^{d/2}} t^{-d/2+1}
\phi\big(\frac{|x|^2}{t}\big) \ip{\cdot}{\cdot}\\
&= t^{-2} p_t(x) \ip{x}{\cdot} \ip{x}{\cdot}
-t^{-1} p_t(x) \ip{\cdot}{\cdot}.
\end{align*}

The following estimate is of central importance for most of the
results of this paper.
\begin{thm}     \label{thm:bd1}
  Let $k \in \enne$, $\ell \geq 0$, $p,\,q \in [1,\infty]$, and $t>0$. If
  $f \in L_p(\erre^d)$ is such that $y \mapsto |y|^\ell f(y) \in
  L_q(\erre^d)$, then, setting
  \[
  \|f\|_{p,q,\ell} = \|f\|_{L_p(\erre^d)} 
  + \bigl\| |\cdot|^\ell f \bigr\|_{L_q(\erre^d)},
  \]
  one has
  \[
  |x|^\ell \bigl\| D^k P^Y_t f(x) \bigr\|_{\LL_k(\erre^d)} \lesssim
  \|f\|_{p,q,\ell} \Bigl( \E T_t^{\frac{\ell-k}{2} - \frac{d}{2p}}
    + \E T_t^{-\frac{k}{2} - \frac{d}{2q}} \Bigr)
  \qquad \forall x \in \erre^d.
  \]
\end{thm}
\begin{proof}
  Taking the norm in $\LL_k(\erre^d)$ on both sides of
  \eqref{eq:Fre-hk} yields
  \begin{equation}\label{lm-eq2}
    \bigl\| D^k p_t(x) \bigr\|_{\LL_k(\erre^d)}
    \lesssim (t^{-k} |x|^k + t^{-k/2})p_t(x) \qquad \forall x \in \erre^d,
  \end{equation}
  thus also
  \begin{align*}
    \bigl\| D^k P_t^Wf(x) \bigr\|_{\LL_k(\erre^d)} &\leq
    \int_{\erre^d} |f(y)| \, \bigl\| D^kp_t(x-y) \bigr\|_{\LL_k(\erre^d)}\,dy\\
    &\lesssim t^{-k/2} \int_{\erre^d} |f(y)| \, \bigl( 1 +
    t^{-k/2}|x-y|^k \bigr) p_t(x-y)\,dy.
  \end{align*}
  Multiplying both sides by $|x|^\ell$ and using the triangle
  inequality, one gets
  \begin{align*}
    &|x|^\ell \bigl\| D^k P_t^Wf(x) \bigr\|_{\LL_k(\erre^d)}\\
    &\qquad \lesssim t^{-k/2} \int_{\erre^d} |f(y)| \, 
    |x-y|^\ell \bigl( 1 + t^{-k/2}|x-y|^k \bigr) p_t(x-y)\,dy\\
    &\qquad\quad + t^{-k/2} \int_{\erre^d} |f(y)| \, |y|^\ell 
    \bigl( 1 + t^{-k/2}|x-y|^k \bigr) p_t(x-y)\,dy\\
    &\qquad =: t^{-k/2}(I_1 + I_2),
  \end{align*}
  Thanks to H\"older's and Minkowski's inequalities, denoting by $p'
  \in [1,\infty]$ the conjugate exponent of $p$, it holds
  \begin{align*}
    I_1 &= \int_{\erre^d} f(y) \bigl(
    |x-y|^\ell + t^{-k/2}|x-y|^{k+\ell} \bigr) p_t(x-y)\,dy\\
    &\leq \|f\|_{L_p(\erre^d)} \Bigl( \bigl\| |\cdot|^\ell p_t
          \bigr\|_{L_{p'}(\erre^d)}
    + t^{-k/2} \bigl\| |\cdot|^{k+\ell} p_t \bigr\|_{L_{p'}(\erre^d)} \Bigr)\\
    &=: \|f\|_{L_p(\erre^d)} (I_{11} + I_{12}).
  \end{align*}
  Note that one has, by well-known scaling properties of the heat
  kernel $p_t$,
  \begin{align*}
    \bigl\| |\cdot|^\ell p_t \bigr\|_{L_{p'}(\erre^d)} 
    &= t^{-d/2} \biggl( \int_{\erre^d} 
       |x|^{\ell p'} \, p_1^{p'}(t^{-1/2}x)\,dx \biggr)^{1/p'}\\
    &= t^{-d/2} \biggl( \int_{\erre^d} 
       |t^{1/2}y|^{\ell p'} \, p_1^{p'}(y) \, t^{d/2} \,dy \biggr)^{1/p'}\\
    &= t^{\frac{\ell}{2} + \frac{d}{2p'} - \frac{d}{2}} \,
       \bigl\| |\cdot|^\ell p_1 \bigr\|_{L_{p'}(\erre^d)},
  \end{align*}
  hence
  \[
  I_{11} = t^{\frac{\ell}{2} + \frac{d}{2p'} - \frac{d}{2}} \, 
  \bigl\| |\cdot|^\ell p_1 \bigr\|_{L_{p'}(\erre^d)}, 
  \qquad 
  I_{12} = t^{\frac{\ell}{2} + \frac{d}{2p'} - \frac{d}{2}} \, 
  \bigl\| |\cdot|^{k+\ell} p_1 \bigr\|_{L_{p'}(\erre^d)},
  \]
  and
  \[
  I_1 \leq \|f\|_{L_p(\erre^d)} \, 
  t^{\frac{\ell}{2} + \frac{d}{2p'} - \frac{d}{2}} \,
  \Bigl( \bigl\| |\cdot|^\ell p_1 \bigr\|_{L_{p'}(\erre^d)} 
       + \bigl\| |\cdot|^{k+\ell} p_1 \bigr\|_{L_{p'}(\erre^d)} \Bigr).
  \]
  Similarly, one has
  \begin{align*}
  I_2 &\leq \bigl\| |\cdot|^\ell f \bigr\|_{L_q(\erre^d)}
  \Bigl( \|p_t\|_{L_{q'}(\erre^d)} 
         + t^{-k/2} \bigl\| |\cdot|^k p_t \bigr\|_{L_{q'}(\erre^d)}
  \Bigr)\\
  &= \bigl\| |\cdot|^\ell f \bigr\|_{L_q(\erre^d)} \,
  t^{\frac{d}{2q'}-\frac{d}{2}} \, \Bigl( \|p_1\|_{L_{q'}(\erre^d)}
  + \bigl\||\cdot|^k p_1\bigr\|_{L_{q'}(\erre^d)} \Bigr).
  \end{align*}
  Collecting estimates, we are left with
  \[
  |x|^\ell \bigl\| D^k P_t^Wf(x) \bigr\|_{\LL_k(\erre^d)} \lesssim
  N_1 \|f\|_{L_p(\erre^d)} t^{\frac{\ell-k}{2} + \frac{d}{2p'} - \frac{d}{2}}
  + N_2 \bigl\| |\cdot|^\ell f \bigr\|_{L_q(\erre^d)} 
  t^{-\frac{k}{2} + \frac{d}{2q'}-\frac{d}{2}},
  \]
  which implies
  \begin{align*}
    &|x|^\ell \bigl\| D^k P_t^Y f(x) \bigr\|_{\LL_k(\erre^d)}\\
    &\hspace{4em}\leq \int_0^\infty |x|^\ell \,
          \bigl\| D^k P_s^W f(x) \bigr\|_{\LL_k(\erre^d)}\,\nu_t(ds)\\
    &\hspace{4em}\lesssim \Bigl( \|f\|_{L_p(\erre^d)} 
    + \bigl\| |\cdot|^\ell f \bigr\|_{L_q(\erre^d)} \Bigr)
    \int_0^\infty \bigl( s^{\frac{\ell-k}{2} + \frac{d}{2p'} - \frac{d}{2}}
    + s^{-\frac{k}{2} + \frac{d}{2q'}-\frac{d}{2}} \bigr)\,\nu_t(ds)\\
    &\hspace{4em}= \Bigl( \|f\|_{L_p(\erre^d)} 
    + \bigl\| |\cdot|^\ell f \bigr\|_{L_q(\erre^d)} \Bigr)
    \Bigl( \E T_t^{\frac{\ell-k}{2} - \frac{d}{2p}}
    + \E T_t^{-\frac{k}{2} - \frac{d}{2q}} \Bigr).
  \qedhere
  \end{align*}
\end{proof}

Taking $p=q=\infty$, one immediately has the following regularizing
property.
\begin{coroll}     \label{cor:bd1}
  Let $k \in \enne$, $\ell \geq 0$, and $t>0$. If $f \in
  \B_{b,c}(\erre^d)$, then, setting
  \[
  M_\ell = \|f\|_\infty + \sup_{y\in\erre^d} |y|^\ell |f(y)|,
  \]
  one has
  \[
  |x|^\ell \bigl\| D^k P^Y_t f(x) \bigr\|_{\LL_k(\erre^d)} \lesssim
  M_\ell \, \E [T_t^{-k/2}+T_t^{(\ell-k)/2}]
  \qquad \forall x \in \erre^d.
  \]
\end{coroll}
As a further immediate consequence of the previous theorem (taking
$p=q=\infty$ and $\ell=0$) we obtain a sufficient condition for the
semigroup associated to a L\'evy process, obtained by subordination of
a Wiener process, to be $k$-smoothing.
\begin{coroll}     \label{cor:smoothY}
  Let $k \in \enne$ and $t>0$. If
  \begin{equation}
    \E T_t^{-k/2} < \infty , \label{1-1}
  \end{equation}
  then $P_t^Y$ is $k$-smoothing, i.e. $P_t^Yf \in C^k_b(\erre^d)$ for
  any $f \in \B_b(\erre^d)$.
\end{coroll}
Similarly, taking $\ell=0$ and $p=q$, one gets sufficient conditions
for the semigroup $P^Y$ to be $L_p$-strong Feller, $1 \leq p \leq
\infty$, or, more generally, to map $L_p(\erre^d)$ into
$C^k_b(\erre^d)$.
\begin{coroll}     \label{cor:Lippi}
  Let $p \in [1,\infty]$, $k \in \enne_0$ and $t>0$. If
  \[
  \E T_t^{-\frac12(k+d/p)} < \infty,
  \]
  then $P_t^Yf \in C^k_b(\erre^d)$ for all $f \in L_p(\erre^d)$. In
  particular, if $\E T_t^{-d/2p}<\infty$, then $P_t^Y$ is $L_p$-strong
  Feller.
\end{coroll}
This result can be extended to (positive) real values of $k$, in which
case the space of differentiable functions $C^k_b$ has to replaced by
H\"older spaces.
\begin{prop}     \label{prop:poldo}
  Let $p \in [1,\infty]$, $\beta>0$ real, and $t>0$. If
  \[
  \E T_t^{-\frac12(\beta+d/p)} < \infty,
  \]
  then $P_t^Yf \in C^\beta_b(\erre^d)$ for all $f \in L_p(\erre^d)$.
\end{prop}
\begin{proof}
  It is easily seen that it is enough to consider the case $0 < \beta
  < 1$. Denoting the conjugate exponent of $p$ by $p'$, one has
  \begin{align*}
  &\frac{\bigl| P_t^Wf(x_1)-P_t^Wf(x_2) \bigr|}{|x_1-x_2|^\beta}\\
  &\hspace{3em}\leq \int_{\erre^d} |f(y)| 
      \frac{|p_t(x_1-y)-p_t(x_2-y)|}{|x_1-x_2|^\beta}\,dy\\
  &\hspace{3em}= \int_{\erre^d} |f(y)| \, \left| 
      \frac{p_t(x_1-y)-p_t(x_2-y)}{(x_1-y)-(x_2-y)}\right|^\beta
      |p_t(x_1-y)-p_t(x_2-y)|^{1-\beta}\,dy\\
  &\hspace{3em}\leq \|f\|_{L_p} \left( \int_{\erre^d}
  \left| \frac{p_t(x_1-y)-p_t(x_2-y)}{(x_1-y)-(x_2-y)}\right|^{\beta p'}
      |p_t(x_1-y)-p_t(x_2-y)|^{(1-\beta)p'}\,dy
  \right)^{1/p'}
  \end{align*}
  Moreover, recalling the scaling properties of $p_t$, one has
  \begin{align*}
    \frac{|p_t(x)-p_t(y)|}{|x-y|} &=
    t^{-d/2} \frac{|p_1(t^{-1/2}x)-p_1(t^{-1/2}y)|}{|x-y|}\\
    &= t^{-d/2} t^{-1/2}
    \frac{|p_1(t^{-1/2}x)-p_1(t^{-1/2}y)|}{|t^{-1/2}x-t^{-1/2}y|}\\
    &\leq t^{-d/2} t^{-1/2} \|p_1\|_{\dot{C}^{0,1}},
  \end{align*}
  where $\|p_1\|_{\dot{C}^{0,1}}$ stands for the Lipschitz constant of
  $p_1$. The latter implies
  \begin{multline*}
    \frac{\bigl| P_t^Wf(x_1)-P_t^Wf(x_2) \bigr|}{|x_1-x_2|^\beta}
    \\\leq \|f\|_\infty \, \bigl\|p_1\bigr\|^\beta_{\dot{C}^{0,1}} \,
    t^{-\beta/2} t^{-d\beta/2} \left( \int_{\erre^d}
      |p_t(x_1-y)-p_t(x_2-y)|^{(1-\beta)p'}\,dy \right)^{1/p'}.
  \end{multline*}
  Again by the scaling properties of $p_t$, as well as elementary
  inequalities, and changing variable a few times, one gets
  \begin{align*}
    \left( \int_{\erre^d} |p_t(x_1-y)-p_t(x_2-y)|^{(1-\beta)p'}\,dy
    \right)^{1/p'}
    &\lesssim \left( \int_{\erre^d} |p_t(y)|^{(1-\beta)p'}\,dy \right)^{1/p'}\\
    &\leq t^{-d(1-\beta)/2} \, t^{\frac{d}{2}\frac{1}{p'}} \bigl\|
    |p_1|^{1-\beta} \bigr\|_{L_{p'}}.
  \end{align*}
  We have thus obtained
  \[
  \sup_{x_1 \neq x_2}
  \frac{\bigl| P_t^Wf(x_1)-P_t^Wf(x_2) \bigr|}{|x_1-x_2|^\beta}
  \lesssim \|f\|_{L_p(\erre^d)} \, t^{-\frac{\beta}{2} - \frac{d}{2p}}.
  \]
  By subordination and Minkowski's inequality, this implies
  \[
  \bigl\| P_t^Y f\bigr\|_{\dot{C}_b^\beta} :=  \sup_{x_1 \neq x_2}
  \frac{\bigl| P_t^Yf(x_1)-P_t^Yf(x_2) \bigr|}{|x_1-x_2|^\beta}
  \lesssim \|f\|_{L_p(\erre^d)} \, \E T_t^{-\frac12(\beta - d/p)}.
  \]
  Since $\E T_t^{-d/2p} \leq \E T_t^{-\frac12(\beta - d/p)}<\infty$,
  the previous corollary yields $P_t^Yf \in C_b$, allowing us to
  conclude that $\bigl\| P_t^Y f\bigr\|_{C_b^\beta}= \|P_t^Yf\|_\infty
  + \bigl\| P_t^Y f\bigr\|_{\dot{C}_b^\beta}<\infty$, which is the
  desired result.
\end{proof}

We shall see in \S\ref{ssec:Bess} (in particular, cf. Corollary
\ref{cor:BH}) that the last corollary can be obtained, at least for
$1<p<\infty$, also by results on embeddings of Sobolev spaces.

\begin{rmk}
  Theorem \ref{thm:bd1}, thus also its corollaries, continue to hold
  also in the more general case that the covariance matrix of $W$ is
  $Q \neq I$ with $\det Q \neq 0$. In fact, in this case the density
  of the Gaussian random variable $W_t$, $t>0$, is
  \[
  p_t^Q(x) = \frac{1}{t^{d/2}} \frac{1}{\sqrt{\det Q}}
  \frac{1}{(2\pi)^{d/2}} \exp\Bigl( \frac{\ip{Q^{-1}x}{x}}{2t} \Bigr),
  \]
  which can be written as
  \[
  p_t^Q(x) = \frac{1}{\sqrt{\det Q}} \, p_t(Q^{-1/2}x).
  \]
  One can now establish a corresponding version of Lemma
  \ref{lm:Hermite}, e.g. introducing the following equivalent scalar
  product and norm in $\erre^d$:
  \[
  \ip{x}{y}_Q := \ip{Q^{-1}x}{y}, \qquad \|x\|^2 := \ip{x}{x}_Q,
  \]
  and computing the Fr\'echet derivatives under the ``new'' topology
  (recall that Fr\'echet differentiability does not depend on the
  metric properties of the underlying space). It is easily seen that
  an expression completely analogous to \eqref{eq:Fre-hk} still holds,
  if one uses the scalar product $\ip{\cdot}{\cdot}_Q$ in place of the
  natural one. One gets, for instance,
  \[
  D^2 p^Q_t(x) = t^{-2} p^Q_t(x) \ip{Q^{-1}x}{\cdot}
  \ip{Q^{-1}x}{\cdot} -t^{-1} p^Q_t(x) \ip{Q^{-1}\cdot}{\cdot}.
  \]
  Moreover, since the norms $|\cdot|$ and $\|\cdot\|$ are equivalent,
  the estimate \eqref{lm-eq2} continues to hold also if $p_t$ is
  replaced by $p_t^Q$.
\end{rmk}

We are now going to provide a simple, yet very useful, sufficient
and necessary condition for the finiteness of negative moments of
subordinators in terms of their Laplace exponent.
\begin{prop}     \label{prop:B1}
  Let $1 \leq p < \infty$ and $t>0$. Then $\E T_t^{-p}<\infty$ if and
  only if $\lambda \mapsto \lambda^{p-1} e^{-t\Phi(\lambda)} \in L_1$.
\end{prop}
\begin{proof}
  By definition of gamma function, that is
  \[
  \Gamma(p) = \int_0^\infty z^{p-1} e^{-z}\,dz,
  \]
  one gets, by the change of variable $z=as$, $a>0$,
  \[
  a^{-p} = \frac{1}{\Gamma(p)} \int_0^\infty s^{p-1} e^{-as}\,ds.
  \]
  This implies, by Tonelli's theorem,
  \[
  \E T_t^{-p} = \frac{1}{\Gamma(p)} 
  \int_0^\infty \lambda^{p-1} \E e^{-\lambda T_t}\,d\lambda
  = \frac{1}{\Gamma(p)} \int_0^\infty \lambda^{p-1} e^{-t\Phi(\lambda)}\,d\lambda,
  \]
  thus finishing the proof.
\end{proof}

It is also possible to give a sufficient condition for the finiteness
of negative moments of subordinators in terms of their L\'evy measure.
The following proposition is a special case of a Tauberian theorem due
to Bismut \cite[Thm.~4.15, p.~208]{Bism-jumps}.
\begin{prop}
  Let $0 < p < 1$, $C>0$. If
  \begin{equation}     \label{eq:taube}
  m\big( ]x,+\infty[ \big) \sim Cx^{-p} \qquad \text{as } x \to 0^+,
  \end{equation}
  then
  \[
  \Phi(\lambda) \sim -C\Gamma(1-p) \lambda^{-p} \qquad \text{as }
  \lambda \to +\infty.
  \]
  In particular, (\ref{eq:taube}) implies $\E T_t^{-p} < \infty$ for
  all $t>0$.
\end{prop}

As a noteworthy application of the criteria just proved we recover a
(known) result on smoothness of $\alpha$-stable densities.
\begin{coroll}     \label{cor:smal}
  Let $Y$ be an $\erre^d$-valued rotation-invariant $\alpha$-stable
  process. Then, for any $f \in \B_b(\erre^d)$, one has $P_t^Y f \in
  C^\infty_b(\erre^d)$ for all $t>0$.
\end{coroll}
\begin{proof}
  One can write $Y=W \circ T$ in distribution, where $T$ is an
  $\alpha/2$-stable subordinator (see e.g. \cite{ST}). In particular
  one has $\Phi(\lambda) \propto \lambda^{\alpha/2}$, hence
  $e^{-t\Phi(\cdot)}$ is rapidly decreasing for all $t>0$ and $\lambda
  \mapsto \lambda^{p-1}e^{-t\Phi(\lambda)} \in L_1$ for all $p \geq
  1$. The desired result then follows by Corollary \ref{cor:smoothY}
  and Proposition \ref{prop:B1}.
\end{proof}

\begin{ex}[Variance-gamma processes]     \label{ex:vg}
  Let $Y=W \circ T$, where $T$ is a Gamma process, independent of $W$,
  with parameters $a$ and $b$. Such a process $Y$ is often called
  (especially in the literature on mathematical finance) a
  variance-gamma process. It is known (see e.g. \cite[p.~73]{Bertoin})
  that, for any $t>0$, the random variable $T_t$ admits the density
  \[
  x \mapsto \frac{b^{at}}{\Gamma(at)} x^{at-1} e^{-bx}, \qquad x \geq
  0.
  \]
  Elementary calculations based on the definition and properties of
  the Gamma function yield that $\E T_t^{-k/2}<\infty$ if and only if
  $t>k/(2a)$. Therefore, according to Corollary \ref{cor:smoothY},
  $P_t^Y$ is $k$-smoothing for all $t>k/(2a)$ (and of course it is
  strong Feller for all $t>0$). As we shall see in the next section,
  the semigroup associated to the variance-gamma process $Y$ is indeed
  only ``eventually'' regularizing.
\end{ex}

\subsection{Smoothing in scales of Sobolev and Bessel spaces}
\label{ssec:Bess}
The results of this subsection are not used in the rest of the
paper. We have nonetheless included them here because they are, in our
opinion, an interesting complement to the smoothing properties of the
semigroup $P^Y$ proved above.

We first consider mapping properties of $P^Y$ from $L_p(\erre^d)$ to
integer-order Sobolev spaces $W_p^m(\erre^d)$. Note that also the
endpoint cases $p=1$ and $p=\infty$ are included.
\begin{prop}
  Let $1 \leq p \leq \infty$, $f \in L_p(\erre^d)$, and $m \in
  \enne_0$. If $\E T_t^{-m/2} < \infty$, then $P_t^Yf \in
  W_p^m(\erre^d)$ with
  \[
  \bigl\| \partial^\alpha P_t^Y \bigr\|_{L_p(\erre^d)} 
  \lesssim \|f\|_{L_p(\erre^d)} \E T_t^{-m/2}
  \]
  for all multiindices $\alpha$ such that $|\alpha|=m$.
\end{prop}
\begin{proof}
  For any multiindex $\alpha$, Lemma \ref{lm:mix} and the scaling
  properties of the heat kernel $p_t$ yield
  \begin{align*}
    \partial^\alpha P^W_tf &= t^{-|\alpha|/2} (-1)^{|\alpha|}
    \int_{\erre^d} f(y) H_\alpha\bigl(t^{-1/2}(x-y)\bigr)p_t(x-y)\,dy\\
    &= t^{-|\alpha|/2} t^{-d/2} (-1)^{|\alpha|} \int_{\erre^d} f(y)
    H_\alpha\bigl(t^{-1/2}(x-y)\bigr)
    p_1\bigl(t^{-1/2}(x-y)\bigr)\,dy,
  \end{align*}
  hence, by Young's inequality and the change of variable formula, one
  obtains
  \begin{align*}
    \bigl\| \partial^\alpha P^W_tf \bigr\|_{L_p(\erre^d)} &\leq
    t^{-|\alpha|/2} t^{-d/2} \|f\|_{L_p(\erre^d)}
    \int_{\erre^d} \bigl|H_\alpha(t^{-1/2}x)\bigr| p_1(t^{-1/2}x)\,dx\\
    &= t^{-|\alpha|/2} \|f\|_{L_p(\erre^d)} \int_{\erre^d} \bigl|H_\alpha(x)\bigr|
    p_1(x)\,dx.
  \end{align*}
  We have thus shown that one has
  \[
  \bigl\| \partial^\alpha P^W_tf \bigr\|_{L_p} \lesssim \|f\|_{L_p} \,
  t^{-|\alpha|/2},
  \]
  which also implies, for any $\alpha$ such that $|\alpha|=m$,
  \[
  \bigl\| \partial^\alpha P^Y_tf \bigr\|_{L_p} \lesssim
  \|f\|_{L_p} \, \E T_t^{-m/2}.
  \]
  In fact, by Minkowski's inequality, one has
  \begin{align*}
    \bigl\| \partial^\alpha P^Y_tf \bigr\|_{L_p} &= \left\|
      \int_0^\infty \partial^\alpha P_s^W f \,\nu_t(ds) \right\|_{L_p}
    \leq
    \int_0^\infty \bigl\|\partial^\alpha P_s^W f\bigr\|_{L_p} \,\nu_t(ds)\\
    &\lesssim \|f\|_{L_p} \int_0^\infty s^{-|\alpha|/2} \,\nu_t(ds) =
    \|f\|_{L_p} \, \E T_t^{-|\alpha|/2}.
  \end{align*}
  This immediately implies that $P_t^Yf \in W_p^m$ by virtue of the 
  well-known estimate $\|\phi\|_{W_p^m} \lesssim \|\phi\|_{L_p} +
  \sum _{|\alpha| = m}\|\partial^\alpha\phi\|_{L_p}$.
\end{proof}

An analogous result can be obtained in the scale of Bessel potential
spaces $H_p^r(\erre^d)$, $1 < p < \infty$, $r \geq 0$ real, thus
generalizing the previous result, even though the endpoint cases $p=1$
and $p=\infty$ are not included. For the proof we need to recall some
facts about analytic semigroups.
Let $-A$ be a linear operator on a Banach space $E$, generating an
analytic semigroup of contractions $S_t=e^{-tA}$, $t>0$ (see
e.g. \cite{Pazy} for details). Then one has, for any $\alpha \geq 0$ and
$t>0$,
\begin{equation}     \label{eq:frest}
\| A^\alpha S(t)f \|_E \lesssim \frac{1}{t^\alpha} \|f\|_E.
\end{equation}
Letting $p \in ]1,\infty[$, $E=L_p$, and $A=-\Delta$, one can show
that existence of negative moments of the subordinator $T_t$ implies
that $P^Y_t$ maps $L_p$ to a Bessel potential space. 
\begin{prop}
  Let $1<p<\infty$. Assume that there exist $r>0$ and $t>0$ such that
  $\E T_t^{-r/2}<\infty$. Then $P_t^Y(L_p) \subseteq H_p^r$, with
  \[
  \bigl\| P_t^Y \bigr\|_{L_p \to \dot{H}_p^r} \lesssim \E T_t^{-r/2}.
  \]
\end{prop}
\begin{proof}
  Since $-\frac12\Delta$ is the generator of $P^W$, which is an analytic
  semigroup of contractions in $L_p$, estimate \eqref{eq:frest} reads
  \[
  \bigl\| (-\Delta)^{r/2} P^W_tf \bigr\|_{L_p} 
  \lesssim \frac{1}{t^{r/2}} \|f\|_{L_p},
  \]
  which in turn implies
  \begin{align*}
    \bigl\| P_t^Yf \bigr\|_{\dot{H}_p^r} &= 
    \bigl\| (-\Delta)^{r/2} P^Y_t f \bigr\|_{L_p} = \Bigl\|
    (-\Delta)^{r/2} \int_0^\infty P^W_s f\,\nu_t(ds) \Bigr\|_{L_p}\\
    &\leq 
    \int_0^\infty \bigl\| (-\Delta)^{r/2} P^W_s f \bigr\|_{L_p}\,\nu_t(ds)\\
    &\lesssim \|f\|_{L_p} \int_0^\infty \frac{1}{s^{r/2}}\,\nu_t(ds)
    = \|f\|_{L_p} \, \E T_t^{-{r/2}}.
  \end{align*}
  The proof is concluded, upon recalling \eqref{eq:homo} and that the
  semigroup associated to any L\'evy process is contracting in $L_p$
  (the latter fact follows easily by Young's inequality for
  convolutions).
\end{proof}

As a consequence, we partially recover the result of Proposition
\ref{prop:poldo} on mapping properties of $P^Y$ from $L_p$ spaces to
H\"older spaces (note that the endpoint cases $p=1$ and $p=\infty$ are
excluded).
\begin{coroll}     \label{cor:BH}
  Let $1<p<\infty$, $f \in L_p(\erre^d)$, and assume that $\E
  T_t^{-r/2}<\infty$ for some $r\in\erre$ such that
  $\sigma:=r-d/p>0$. Then $P_t^Yf \in C_b^\sigma(\erre^d)$.
\end{coroll}
\begin{proof}
  If $\sigma \not\in \enne$, the result follows by Sobolev embedding
  (Lemma \ref{lm:Bess}), and by Corollary \ref{cor:Lippi} if $\sigma
  \in \enne$.
\end{proof}

\begin{rmk}
  If $Y$ is a rotationally invariant $\alpha$-stable process on
  $\erre^d$, then the generator of the (analytic) semigroup $P^Y$
  coincides with $(-\Delta)^{\alpha/2}$. Therefore, by well-known
  properties of analytic semigroups, one has $P^Y_t f \in \dom
  ((-\Delta)^{k\alpha})$ for all $k>0$ and $t>0$, i.e. $f \in
  L_p(\erre^d)$, $1<p<\infty$, implies $P^Y_t f \in
  \dot{H}_p^\sigma(\erre^d)$ for all $\sigma>0$ and $t>0$, hence also
  $P^Y_t f \in H_p^\sigma(\erre^d)$ because $P_t^Y$ is contracting in
  $L_p(\erre^d)$. By the Sobolev embedding theorem, this implies that
  $P_t^Yf \in C^\infty_b(\erre^d)$ for all $f \in L_p(\erre^d)$. The
  same conclusion can be reached applying the previous corollary,
  recalling that stable subordinators have finite negative moments of
  all orders (cf. the proof of Corollary \ref{cor:smal}).
\end{rmk}

\subsection{Smoothing properties of $P^Z_t$}
Thus far we have only considered smoothing properties of Markovian
semigroups associated to subordinated Wiener processes. In this
subsection we provide a simple argument which allows to extend the
results of the previous section to a much larger class of semigroups.

In particular, let $Z=Y+\xi$, where $Y=W \circ T$ is a subordinated
Wiener process, and $\xi$ is a further L\'evy process, independent of
$Y$. Then $Z$ is a L\'evy process, and we denote by $P^Z$ its
associated semigroup, defined in the usual way, i.e.
\[
P^Z_tf(x) := \E f(x+Z_t) \equiv \E f(x+Y_t+\xi_t), 
\qquad f\in\B_b(\erre^d). 
\]
This notation will be used throughout this subsectin without further
notice.

\begin{prop} \label{prop:multi-noise}
  Let $t>0$ and assume that $P^Y_t$ is $k$-smoothing. Then $P^Z_t$ is
  also $k$-smoothing. Moreover, if there exists $\beta>0$ such that
  \[
  \bigl\| D^k P^Y_t f (x) \bigr\|_{\LL_k(\erre^d)} \lesssim t^\beta \|f\|_\infty
  \qquad \forall x \in \erre^d,\ \forall f \in \B_b(\erre^d),
  \]
  then the same estimate is satisfied with $P^Y_t$ replaced by $P^Z_t$.
\end{prop}
\begin{proof}
  Let $f\in \B _b(\erre ^d)$ and $\alpha$ be a multi-index such that
  $|\alpha | \leq k$.  Since $\xi$ is independent of $Y$, we have
  \begin{align*}
    \partial^\alpha P^Z_t f(x)
    &= \partial^\alpha \E f(x + Y_t + \xi_t)\\
    &= \partial^\alpha \E\, \E \bigl[f(x + Y_t + \xi_t) \big| \xi_t\bigr]\\
    &= \partial^\alpha \int_{\erre^d} P_t^Yf(x+y)\,\mu_t(dy),
  \end{align*}
  where $\mu_t$, the distribution of the random variable $\xi_t$, is a
  probability measure on $\erre^d$. Since $P^Y_t f \in
  C^k_b(\erre^d)$, the dominated convergence theorem implies that one
  can differentiate under the integral sign, obtaining
  \begin{equation}\label{eq:multi-noise1}
    \partial^\alpha P^Z_t f(x) = \int_{\erre^d} 
    \bigl[\partial^\alpha P^Yf\bigr](x+y)\,\mu_t(dy).
  \end{equation}
  Let $(x_n)_{n\in\enne}$ a sequence converging to $x$ in
  $\erre^d$. The dominated convergence theorem yields
  \[
  \lim_{x_n \to x} \int_{\erre^d} \bigl[ \partial^\alpha P_t^Yf
  \bigr](x_n+y)\,\mu_t(dy) = \int_{\erre^d} \bigl[ \partial^\alpha
  P_t^Yf \bigr](x+y)\,\mu_t(dy),
  \]
  hence, by \eqref{eq:multi-noise1}, that $\partial^\alpha P_t^Zf$ is
  continuous, or, equivalently, that $P_t^Zf \in
  C^k(\erre^d)$. Moreover, by Minkowski's inequality, one has
  \begin{equation}     \label{eq:minkia}
  \bigl\| \partial^\alpha P_t^Zf \bigr\|_\infty \leq \int_{\erre^d}
  \bigl\| \partial^\alpha P_t^Yf \bigr\|_\infty \,\mu_t(dy) =
  \bigl\| \partial^\alpha P_t^Yf \bigr\|_\infty,
  \end{equation}
  which implies $\partial^\alpha P_t^Zf \in C_b(\erre^d)$, hence also
  $P_t^Zf \in C_b^k(\erre^d)$.
  The second assertion is an obvious consequence of \eqref{eq:minkia}.
\end{proof}

Thanks to the previous proposition, we have the following smoothing
result, generalizing Corollary \ref{cor:bd1}. 

\begin{coroll}\label{cor:smoothY2}
  Let $k \in \enne$ and $t>0$. Assume that $Z=Y+\xi$, where $Y=W
  \circ T$ and
  \begin{equation*}
    \E T_t^{-k/2} < \infty.
  \end{equation*}
  Then $P_t^Z$ is $k$-smoothing, i.e. $P_t^Zf \in C^k_b(\erre^d)$ for
  any $f \in \B_b(\erre^d)$.  Moreover, one has 
  \[
  \bigl\| D^k P_t^Z f (x)\bigr\|_{\LL_k(\erre^d)} \lesssim \|f\|_{\infty}
  \E T_t^{-k/2} \qquad \forall x \in \erre^d,\ \forall f \in \B_b(\erre^d).
  \]
\end{coroll}
The above argument holds even if $\xi$ is not Markovian. Let $\tilde
Z^x_t = x + Y_t + \eta _t$ where $Y$ is a subordinated Wiener process
as above and $\eta$ is \emph{any} stochastic process independent of
$Y$ such that $\eta _0 =0$ almost surely (in particular it is not
necessary to assume that $\eta$ is a L\'evy nor a Markov process). Let
us define, for $t\geq 0$, define a bounded linear operator $A^{\tilde
  Z}_t$ on $\B_b(\erre ^d)$ by
\[
A^{\tilde Z}_t f(x) := \E f(\tilde Z^x_t) = \E f(x + Y_t + \eta _t),
\qquad f\in \B_b(\erre ^d).
\]
We remark that $t \mapsto A^{\tilde Z}_t$ is, in general, not a
semigroup. Nonetheless, the above results still hold in this case,
with exactly the same proofs, and we have the following proposition.
\begin{prop}\label{prop:smoothY3}
  Let $k \in \enne$ and $t>0$. Assume that
  \begin{equation*}
    \E T_t^{-k/2} < \infty.
  \end{equation*}
  Then $A^{\tilde Z}_t$ is $k$-smoothing. Moreover, one has 
  \[
  \bigl\| D^k A^{\tilde Z}_t f (x)\bigr\|_{\LL_k(\erre^d)} \leq C
  \|f\|_{\infty}\E T_t^{-k/2} \qquad \forall x\in \erre ^d,\ 
  \forall f\in B_b(\erre^d),
  \]
  where $C$ is a constant depending only on $d$ and $k$.
\end{prop}

\begin{rmk}
  Proposition \ref{prop:smoothY3} implies that smoothing properties
  holds for $A^{\tilde Z}_t$, even if the driving noise is not L\'evy
  process. For example, let $\xi$ be a $\erre ^d$-valued fractional
  Brownian motion. Then, the driving noise is no more L\'evy noise,
  but the estimate in Proposition \ref{prop:smoothY3} holds.
\end{rmk}

\subsection{On smoothing properties of general L\'evy
processes}     \label{ssec:smoothL}
It is well known (and it was used in the proof of Lemma \ref{lm:dens})
that the operation of convolution with an $L_1$ function is strong
Feller. In fact the converse result is true as well, as it was proved
by Hawkes \cite{Hks} (the result was actually already proved, using
different terminology, by Brainerd and Edwards \cite{BrEd}). We state
their result, and provide the (short) proof for completeness.
\begin{prop}[Brainerd and Edwards, Hawkes]
  Let $\mu$ be a finite measure on $\erre^d$, and consider the
  linear operator $A_\mu$ defined as
  \[
  [A_\mu f](x) = \int_{\erre^d} f(x-y)\,\mu(dy).
  \]
  Then the following assertions are equivalent:
  \begin{itemize}
  \item[\emph{(a)}] $\mu$ is absolutely continuous with respect to
    Lebesgue measure;
  \item[\emph{(b)}] $A_\mu$ is strong Feller;
  \item[\emph{(c)}] $A_\mu$ is $c$-strong Feller.
  \end{itemize}
\end{prop}
\begin{proof}
  (a) implies (b): let us denote, with an harmless abuse of notation, the density of $\mu$ again by $\mu$. Then $\mu \in
  L_1$, hence $A_\mu f = f \ast \mu$, and, as already recalled above,
  $A_\mu$ maps $\B_b$ to $C_b$. Obviously (b) implies (c). In order to
  conclude, we only have to show that (c) implies (a). Define the
  measure $\tilde{\mu}$ by $B \mapsto \tilde{\mu}(B):=\mu(-B)$, and
  observe that the linear operator
  \[
  [\tilde{A}_\mu f](x) := \int_{\erre^d} f(x-y)\,\tilde{\mu}(dy)
  \]
  is the formal adjoint of $A_\mu$, in the sense that, for any $f$, $g
  \in \B_b$, one has
  \[
  \int_{\erre^d} A_\mu f\,g = \int_{\erre^d} f\,\tilde{A}_\mu g
  \]
  Assume that $B \subset \erre^d$ has zero Lebesgue measure, and write
  $B=\cup_{n\in\enne} B_n$, with $B_n:=B \cap E_n$, where
  $\cup_{n\in\enne} E_n=\erre^d$, and $E_n$ is bounded for each
  $n\in\enne$. Let $n \in \enne$ be arbitrary but fixed, and set
  $f:=1_{B_n}$. Then $f \in \B_{b,c}$ and thus, by hypothesis, $A_\mu f \in
  C_b$ and
  \[
  \int_{\erre^d} A_\mu f\,g = 0 \qquad \forall g \in C^\infty_c,
  \]
  hence $A_\mu f(x)=0$ for almost all $x \in \erre^d$ with respect to the Lebesgue measure. In particular,
  \[
  0 = A_\mu f(0) = \int_{\erre^d} 1_{B_n}\,d\mu = \mu(B_n),
  \]
  thus also $\mu(B)=0$ because $B$ is the union of countably many sets
  of $\mu$-measure zero.
\end{proof}
\begin{rmk}
  (i) The assumption that $\mu$ is a finite measure is essential in
  the previous proposition. In fact, if one just assumes that $\mu$ has a
  density in $L_{1,\mathrm{loc}}$, it is easily seen that (a) does not
  imply (b), and that (b) actually implies $\mu \in L_1$ (e.g. choosing
  $f \equiv 1$). Moreover, without assuming that $\mu$ is a finite
  measure, (c) does indeed only imply $\mu \in L_{1,\mathrm{loc}}$.

  \smallskip

  (ii) As a consequence of the above proposition, Hawkes \cite{Hks}
  shows that the semigroup generated by a (finite dimensional) L\'evy
  process is strong Feller if and only if its transition densities are
  absolutely continuous with respect to Lebesgue measure and that, in
  this case, the densities are lower semicontinuous.
\end{rmk}
An immediate, but nonetheless quite useful consequence of the previous
proposition is the following.
\begin{coroll}
  Let $k \in \mathbb{N} \cup \{0\}$. If $\mu$ admits a density
  belonging to $W_1^k$, then $A_\mu$ is $k$-smoothing.
\end{coroll}
Unfortunately, however, it is not possible to assert that if $A_\mu$
is $k$-smoothing, then $\mu \in W_1^k$. This follows immediately from
Theorem \ref{thm:KnoSchi} below.

\medskip

The following result by Knopova and Schilling \cite{KnoSchi}, which
extends a ``classical'' criterion by Hartman and Wintner \cite{HaWi},
says, among other things, that the transition semigroup of a L\'evy
process is $1$-smoothing at all (positive) times if and only if it is
$\infty$-smoothing at all (positive) times.
\begin{thm}[\cite{KnoSchi}]     \label{thm:KnoSchi}
  Let $Y$ be an $\erre^d$-valued L\'evy process without Gaussian
  component, with characteristic function
  $\E\exp\big(it\ip{\xi}{Y_t}\big)=:e^{-t\psi(\xi)}$ and transition
  kernel $\pi_t: B \mapsto \P(Y_t \in B)$. The following assertions
  are equivalent:
  \begin{itemize}
  \item[\emph{(a)}] It holds
    \begin{equation}     \label{eq:HW}
    \lim_{|\xi|\to\infty} \frac{\operatorname{Re}\,\psi(\xi)}{\log
      (1+|\xi|)} = \infty;
    \end{equation}
  \item[\emph{(b)}] $\pi_t$ is absolutely continuous for all $t>0$,
    with density $p^Y_t$ such that $\partial^\alpha p^Y_t \in L_1
    \cap C_0$ for any multiindex $\alpha \geq 0$;
  \item[\emph{(c)}] $\pi_t$ is absolutely continuous for all $t>0$, with
    density $p^Y_t$ such that $\nabla p^Y_t \in L_1$.
  \end{itemize}
  Moreover, if there exists an increasing function $g$ such that
  $\psi(\xi)=g(|\xi|^2)$, then the above assertions are also
  equivalent to the following ones:
  \begin{enumerate}
  \item[\emph{(d)}] $\pi_t$ is absolutely continuous for all $t>0$,
    with density $p^Y_t \in C_0$;
  \item[\emph{(e)}] $\pi_t$ is absolutely continuous for all $t>0$,
    with density $p^Y_t \in L_\infty$;
  \item[\emph{(f)}] $e^{-t\psi} \in L_1$ for all $t>0$.
  \end{enumerate}
\end{thm}
\begin{ex}[Stable processes]
  If $Y$ is an isotropic stable process, then
  $\psi(\xi)=|\xi|^\alpha$, so that (\ref{eq:HW}) is clearly
  satisfied, and we get that the transition density of $Y_t$ belongs
  to $C^\infty_b$ at all positive times, hence $P^Y_t$, $t>0$, is
  infinitely smoothing. We reached exactly the same conclusion by our
  method based on estimates of negative moments of stable
  subordinators, cf. Corollary \ref{cor:smal}.
\end{ex}
\begin{ex}[Variance-gamma processes]
  Let us consider the variance-gamma process $Y$ defined in Example
  \ref{ex:vg}. We have $\E e^{-\lambda T_t} = e^{-t\Phi(\lambda)}$,
  with
  \[
  \Phi(\lambda) = a\log\big( 1+\lambda/b \big)
  \]
  (see e.g. \cite[p.~73]{Bertoin}), thus also
  \[
  \E e^{i\ip{\xi}{W(T_t)}} = e^{-t\psi(\xi)}, \qquad \psi(\xi) =
  \Phi(|\xi|^2) = a\log\big( 1+|\xi|^2/b \big),
  \]
  that is
  \[
  \E e^{i\ip{\xi}{(W(T_t)}} = \big( 1+|\xi|^2/b \big)^{-at}.
  \]
  In particular, we have
  \begin{equation}     \label{eq:HW-vg}
  \liminf_{|\xi|\to\infty} \frac{\operatorname{Re}\,\psi(\xi)}{\log
    (1+|\xi|)} = 2a,
  \end{equation}
  hence the density of the variance-gamma process (when it exists) is
  not in $C^\infty_b$ for all $t>0$.  Let us also recall that Hartman
  and Wintner \cite{HaWi} proved that if there exists $t_0 \geq 0$
  such that
  \[
  \liminf_{|\xi|\to\infty} \frac{\operatorname{Re}\,\psi(\xi)}{\log
    (1+|\xi|)} > \frac{d}{t_0},
  \]
  then $\pi_t$ is absolutely continuous for all $t \geq t_0$ with
  density $p^Y_t \in L_1 \cap C_0$. Therefore \eqref{eq:HW-vg}
  implies that the variance-gamma process $Y$ admits a density in $L_1
  \cap C_0$ for all $t > d/(2a)$. As a matter of fact this condition
  is sharp, as one can verify by a direct calculation: since
  \[
  p^Y_t(x) \propto \int_0^\infty e^{-|x|^2/(2s)}
  s^{at-1-d/2}e^{-bs}\,ds,
  \]
  it is easily seen that the integral is finite for all $x \neq 0$,
  while for $x=0$ it is finite if and only if $t>d/(2a)$. In other
  words, the density of the variance-gamma process has a singularity
  at the origin for $t \leq d/(2a)$. Equivalently, since $\xi \mapsto
  e^{-t\psi(\xi)} \in L_1$ for all $t>d/(2a)$, one could also conclude
  by the Riemann-Lebesgue lemma that $p^Y_t \in L_1 \cap C_0$ for all
  $t>d/(2a)$. Our method using subordination instead yields that $p^Y_t
  \in L_1$ and $\nabla p^Y_t \in L_1$ for all $t>1/(2a)$. Of course
  these properties do not imply that $p^Y_t \in C_0$.
\end{ex}
\begin{rmk}
  The discussion of the variance-gamma process in the above example
  can be generalized in a rather straightforward way to geometric
  strictly $\alpha$-stable processes, for which
  $\psi(\xi)=\log(1+|\xi|^\alpha)$. See e.g. \cite[ch.~5]{BoBy} for
  more information about this class of processes.
\end{rmk}

As it will become clear in the next section, the above results, while
powerful and interesting in their own right, seem to be of little help
for establishing smoothing properties of the semigroup generated by
the solution to a SDE driven by a L\'evy process. The main obstruction
is of course that the transition kernels of the solution to a SDE are
not translation invariant. Moreover the law of the solution is not
infinitely divisible, with the exception of very simple
situations. Even restricting our attention to proving smoothing
properties of semigroups generated by L\'evy processes, our approach
through subordination and negative moments of the subordinator is, in
general, not comparable to the criteria quoted above. In fact, while
we can cover only a particular class of L\'evy processes, our results
give smoothing estimates that can depend on time. Most importantly,
our method gives explicit estimates on the rate of blow-up as $t \to
0$ of the norm of $P_t^Yf$, $f \in \B_b$, in spaces of type
$C^k_b$. As we shall see, this is essential for the developments in
the next section.


\section{Smoothing for the SDE (\ref{eq:sde})}
\label{sec:main}
The main result of this section is the following theorem, where we
establish the strong Feller property of the semigroup associated to
the solution of an SDE driven by a subordinated Wiener process,
assuming that the subordinator satisfies a suitable integrability
condition.

Throughout this section we shall tacitly assume that $P_t^X$ is Feller
for all $t \geq 0$. As is well known, this is always the case if the
solution to (\ref{eq:sde}) depends continuously on the initial datum
$x$. This condition is satisfied very often, e.g. when the drift term
$b$ is Lipschitz continuous, or, more generally, when $b$ is
continuous, dissipative, and with polynomial growth.

We first consider the case that $\xi=0$, so that the noise is a L\'evy
process obtained as subordination of a Wiener process.
\begin{thm} \label{thm:drift}
  Let $\ell \geq 0$. Assume that $\xi =0$, there exists
  $\delta>0$ such that
  \begin{equation}     \label{eq:11}
    \int_0^\delta \E\bigl( T_s^{-1/2} + T_s^{(\ell-1)/2} \bigr) \,ds < \infty,
  \end{equation}
  and $x \mapsto b(x)(1+|x|)^{-\ell} \in C_b(\erre^d)$. Then $P^X_t$
  is strong Feller for all $t>0$.
\end{thm}
\begin{proof}
  Let $f \in \B_{b,c}$ and $t>0$. That $x \mapsto P_t^Xf(x)$ is
  bounded is immediate by the stochastic representation $P_t^Xf(x)=\E
  f(X_t^x)$.  Therefore we just have to prove that $P_t^Xf$ is
  continuous. Assumption \eqref{eq:11} implies that $\E T_s^{-1/2}\,ds
  < \infty$ for all $s \in ]0,\delta[$, therefore, by Corollary
  \ref{cor:smoothY}, $P_s^Y$ is $1$-smoothing for all $s \in
  ]0,\delta[$. In particular, one has $DP_s^Yf \in C_b$ and, by virtue
  of Theorem \ref{thm:bd1},
  \begin{equation}     \label{eq:33}
  \bigl\| DP_s^Yf \bigr\|_\infty \lesssim \|f\|_\infty \, \E T_s^{-1/2}
  \end{equation}
  for all $s \in ]0,\delta[$.  For any $t \in ]0,\delta[$,
  we have, by Duhamel's formula,
  \begin{equation}\label{eq:Duhamel}
    P^X_tf = P^Y_tf + \int_0^t P^X_{t-s} \bip{b}{DP^Y_{s}f}\,ds.
  \end{equation}
  Since $P^Y_tf \in C^1_b$, it is enough to prove that the integral on
  the right-hand side is a continuous function. Note that $P^Y_tf \in
  C^1_b$ and $b \in C$ imply $\bip{b}{DP^Y_{s}f} \in C$, hence we have
  that $P_t^X$ is $c$-strong Feller, if we can show that the sup-norm
  of the integral is finite.
  To this purpose, note that we can write
  \begin{align*}
  \bip{b}{DP^Y_{s}f}(x) &= \bip{b(x)(1+|x|)^{-\ell}}{(1+|x|)^\ell DP^Y_sf(x)}\\
  &\lesssim \bip{b(x)(1+|x|)^{-\ell}}{DP^Y_sf(x)}
  + \bip{b(x)(1+|x|)^{-\ell}}{|x|^\ell DP^Y_sf(x)}\\
  &=: I^1_s + I^2_s,
  \end{align*}
  hence also
  \[
  \biggl\| \int_0^t P^X_{t-s} \bip{b}{DP^Y_{s}f}\,ds \biggr\|_\infty 
  \lesssim \int_0^t \bigl\| P^X_{t-s} (I^1_s+I^2_s) \bigr\|_\infty\,ds.
  \]
  Since $b(1+|\cdot|)^{-\ell} \in C_b$, taking \eqref{eq:33} into
  account, and recalling that $P^X_t$ is contracting in $L_\infty$
  because it is Markovian, we obtain
  \[
  \int_0^t \bigl\| P^X_{t-s} I^1_s \bigr\|_\infty\,ds \lesssim
  \|f\|_\infty \, \bigl\| b(1+|\cdot|)^{-\ell} \bigr\|_\infty
  \int_0^t \E T_s^{-1/2}\,ds,
  \]
  which is finite by hypothesis. Analogously, appealing to Theorem
  \ref{thm:bd1}, we have
  \[
  \int_0^t \bigl\| P^X_{t-s} I^2_s \bigr\|_\infty\,ds \lesssim
  \bigl( \|f\|_\infty + \bigl\| |\cdot|^\ell f \bigr\|_\infty \bigr) \,
  \bigl\| b(1+|\cdot|)^{-\ell} \bigr\|_\infty
  \int_0^t \E\bigl(T_s^{-1/2} + T_s^{(\ell-1)/2}\bigr)\,ds,
  \]
  which is finite by hypothesis, recalling that $f \in \B_{b,c}$. We
  have thus established that $P^X_t$ maps $\B_{b,c}$ to $C_b$ for all
  $t \in ]0,\delta[$. Lemma \ref{lm:Feller} implies that $P^X_t$ is
  strong Feller for all $t \in ]0,\delta[$, hence for all $t>0$: in
  fact, if $t>\delta$, one can write
  $P_t^Xf=P^X_{t-\delta/2}P^X_{\delta/2}f$, from which it follows that
  $P^X_tf \in C_b$ because $P_{t-\delta/2}^X$ is Feller and
  $P^X_{\delta/2}$ is strong Feller.
\end{proof}

Here is a result about the $L_p$-strong Feller property of $P^X_t$.
\begin{thm} \label{thm:drift-Lp}
  Let $1 \leq p \leq \infty$. Assume that $\xi =0$, $b \in
  C_b(\erre^d)$ and that there exists $\delta>0$ such that
  \begin{equation}     \label{eq:11.0}
    \int_0^\delta \E T_s^{-\frac12-\frac{d}{2p}} \,ds < \infty.
  \end{equation}
  Then $P^X_tf \in C_b(\erre^d)$ for all $f \in L_p(\erre^d)$.
\end{thm}
\begin{proof}
  The proof is similar to the one of the previous theorem, hence we
  omit some detail. By Corollary \ref{cor:Lippi} one has $DP_s^Yf \in
  C_b$ with
  \[
  \bigl\| DP_s^Yf \bigr\|_\infty \lesssim \|f\|_{L_p} \,
  \E T_s^{-\frac12-\frac{d}{2p}} \qquad \forall s \in ]0,\delta[.
  \]
  For any $t \in ]0,\delta[$, we have, by Duhamel's formula,
  \[
    P^X_tf = P^Y_tf + \int_0^t P^X_{t-s} \bip{b}{DP^Y_{s}f}\,ds.
  \]
  Since $P^Y_tf \in C^1_b$, it is enough to prove that the integral on
  the right-hand side belongs to $C_b$. Since $P^Y_tf \in C^1_b$ and
  $b \in C$ imply $\bip{b}{DP^Y_{s}f} \in C$, it is enough to show
  that the sup-norm of the integral is finite: one has
  \[
  \bigl\| \bip{b}{DP^Y_{s}f} \bigr\|_\infty \lesssim \|b\|_\infty \,
  \|f\|_{L_p} \, \E T_s^{-\frac12-\frac{d}{2p}}
  \qquad \forall s \in ]0,\delta[,
  \]
  hence also, by Minkowski's inequality,
  \[
  \biggl\| \int_0^t P^X_{t-s} \bip{b}{DP^Y_{s}f}\,ds \biggr\|_\infty 
  \lesssim \|b\|_\infty \, \|f\|_{L_p} 
  \int_0^t \E T_s^{-\frac12-\frac{d}{2p}}\,ds,
  \]
  which is finite by assumption \eqref{eq:11.0}. This proves that
  $P^X_t$ maps $L_p$ to $C_b$ for all $t \in ]0,\delta[$, hence also
  for all $t>0$ by the same argument used above.
\end{proof}
\begin{rmk}
  Note that, choosing $f$ equal to zero outside a set of Lebesgue
  measure zero, it is immediately seen that $P_t^Xf_1=P_t^Xf_2$
  everywhere if $f_1=f_2$ almost everywhere.
\end{rmk}

\begin{ex}
  Assume that $T$ is self-similar with self-similarity index $\beta$,
  i.e. $T_t = t^\beta T_1$ in distribution, and $\E
  T_1^{-1/2}<\infty$. Then \eqref{eq:11} certainly holds if $\beta<2$
  and $\ell \in \{0,1\}$. In fact, one has $\E
  T_t^{-1/2}=t^{-\beta/2}\E T_1^{-1/2}$, which is integrable with
  respect to $t$ around zero if (and only if) $\beta<2$. In
  particular, assuming $\ell \in \{0,1\}$, \eqref{eq:11} always holds
  if $T$ is an $\alpha/2$-stable subordinator, which is self-similar
  with index $\alpha/2$, $\alpha < 2$, and whose inverse moments (of
  any order) are finite, as already seen above. This in turn implies
  the strong Feller property for the semigroup generated by the
  solution to an SDE with linearly growing drift (e.g. of
  Ornstein-Uhlenbeck type) driven by a rotationally invariant stable
  process.
\end{ex}

Finally, we state a result for the case that $\xi \neq 0$ in
\eqref{eq:sde}.
\begin{thm}\label{thm:drift2}
Assume that there exists $\delta >0$ such that
\begin{equation}\label{eq:thm-drift2-1}
\int _0 ^\delta \E T_s ^{-1/2} ds < \infty ,
\end{equation}
and $b\in C_b(\erre^d)$. Then, $P^X_t f$ is strong Feller for $t>0$.
\end{thm}
\begin{proof}
  The proof is completely analogous to that of Theorem \ref{thm:drift}
  for $\ell =0$. The only difference is that one has to appeal to
  Proposition \ref{prop:smoothY3} instead of Theorem \ref{thm:bd1}.
\end{proof}

\begin{rmk}
  Theorem \ref{thm:drift2} implies that if $Z$ can be decomposed into
  the independent sum of L\'evy processes $Y$ and $\xi$, if $Y$ is a
  subordinated Wiener process, and if the subordinator satisfies the
  integrable condition of the negative moment (\ref{eq:thm-drift2-1}),
  $P^X_t$ has the strong Feller property for $t>0$.  Here, note that
  nothing is assumed on $\xi$.  This means that the part $Y$ of the
  noise determines the smoothing properties of $P^X_t$.
\end{rmk}


\section{Strong Feller property via Malliavin
  calculus: a special case}
\label{sec:Malliavin}

Unfortunately it does not seem possible to adapt the method of the
previous section to the case of equations with multiplicative noise of
the type
\begin{equation}     \label{eq:fuerte}
dX_t = b(t,X_t)\,dt + \sigma(t,X_{t-})\,dY_t, \qquad X_0=x \in \erre^d,
\end{equation}
essentially because one would need to have quantitative control on the
smoothing properties at small time of the semigroup generated by the
solution to the corresponding SDE without drift.

In the following we obtain the strong Feller property for the
semigroup generated by $X$, by a completely different method. In particular, adapting some techniques based on
Malliavin calculus that were developed in \cite{Sei2}, we consider
equations driven by rotationally-invariant stable processes. Some
smoothness of the coefficients $b$ and $\sigma$ has also to be imposed
(cf.~Theorem \ref{thm3-2} below).

Let us recall that, for SDEs driven by Brownian motions, general
existence and regularity results of transition probability densities
(implying the strong Feller property) are obtained in \cite{KuStr2}. A
crucial role in the argument of \cite{KuStr2} is played by existence
and integrability properties (with respect to the probability measure)
of the stochastic flow. In the case of equations driven by stable
processes, a major obstruction to the extension of this method comes
from the fact that $\alpha$-stable laws have infinite moments of order
$\alpha$ and higher. In \cite{Sei2} a version of Malliavin calculus
for SDEs driven by stable noise is developed, via subordination
techniques, avoiding the problem of integrability of the stochastic
flow.

Before turning to the main result of this section, we provide a
sufficient condition on the transition densities of a Markovian
semigroup to be strong Feller.
\begin{lemma} \label{thm3-1}
  Let $(P_t)_{t>0}$ be a Markovian semigroup on $\B_b(\erre^d)$ such
  that
  \[
  P_t f(x) =\int_{\erre^d} f(y) q_t(x,y)\,dy,
  \]
  for some function $q_t: \erre^d \times \erre^d \to \erre$.  Assume
  that $P_t$ has the Feller property and that for all $M>0$ and
  $x_0 \in {\erre}^d$, there exists $p>1$ and $\gamma>0$ such that
  \begin{equation}
    \label{u-bdd}
    \sup_{x\in {\erre}^d ;\, |x-x_0|<\gamma } 
    \int_{\{ y\in {\erre}^d;\, |y|\leq M \}} q_t(x,y) ^p \,dy 
    < \infty.
  \end{equation}
  Then $P_t$ has the strong Feller property.
\end{lemma}
\begin{proof}
  Let $f\in {\B _{b,c}}(\erre ^d)$ be given. Take $M>0$ such that
  $\operatorname{supp} f \subset \{x\in {\erre}^d ;|x|<M/2\}$.  Let
  $x_0 \in \erre ^d$ given.  Choose $p>1$ and $\gamma >0$ such that
  (\ref{u-bdd}) holds and let
  \[
  C:= \sup _{x\in {\erre}^d ;\, |x-x_0|<\gamma } \int _{\{ y\in
    {\erre}^d;\, |y|\leq M \}} q_t(x,y) ^p dy.
  \]
  Let $p'$ be the conjugate exponent of $p$, i.e. $1/p +1/p' =1$.  For
  any $\varepsilon >0$ there exists $g\in C_c^\infty (\erre ^d)$ such
  that $||f-g||_{p'} < (4C^{1/p})^{-1} \varepsilon$ and
  $\operatorname{supp} g \subset \{x\in {\erre}^d ;|x|<M\}$.  Since
  $P_t$ has the Feller property, there exists $\delta >0$ such that
  $|P_t g(x)- P_t g(x_0)| <\varepsilon /2$ if $|x-x_0| <\delta$.
  Therefore we have, for $|x-x_0| < \min(\delta,\gamma)$,
  \begin{align*}
    & |P_t f(x) - P_t f(x_0)|\\
    & \leq |P_t f(x) - P_t g(x)| + | P_t g(x) - P_t g(x_0)| + |P_t g(x_0) - P_t f(x_0)|\\
    & < \frac{\varepsilon}{2} + \int _{\erre ^d} |f(y)-g(y)| q_t(x,y) dy + \int _{\erre ^d} |f(y)-g(y)| q_t(x_0,y) dy \\
    & \leq \frac{\varepsilon}{2} + 2||f-g||_{p'} \left( \sup _{x\in {\erre}^d ;\, |x-x_0|<\gamma } \int _{\{ y\in {\erre}^d;\, |y|\leq M \}} q_t(x,y) ^p dy \right) ^{1/p} \\
    & < \varepsilon.
  \end{align*}
  Thus, we have the continuity of $P_t f$ at $x_0$.  Since $x_0$ is an
  arbitrary point in $\erre ^d$, we have $P_tf \in C(\erre ^d)$.  The
  boundedness of $P_tf$ follows immediately from the Markov property
  of $P_t$.  Hence, $P_t$ is $c$-strong Feller, and the proof is
  completed thanks to Lemma \ref{lm:Feller}.
\end{proof}
\begin{rmk}
  Lemma \ref{thm3-1} is a criterion for Markovian semigroups to be
  strong Feller, and the criterion is similar to that obtained in
  \cite[Corollary 2.2]{SchiWa}. The advantage of Lemma \ref{thm3-1} is
  that it is applicable to the case that the transition probability
  density is not bounded. For example, the gamma process with certain
  parameters has unbounded transition density, but it satisfies the
  strong Feller property. We also remark that, on the other
  hand, \cite[Corollary~2.2]{SchiWa} would suffice to prove Theorem
  \ref{thm3-2}.
\end{rmk}

Similarly to Lemma \ref{thm3-1}, we have the following sufficient
condition for the $L_p$-strong Feller property.
\begin{coroll}
  Let $t \in (0,\infty)$.  Assume that $P_t$ has the Feller property
  and that for all $x_0\in {\erre}^d$, there exists $p>1$
  and $\gamma >0$ such that
  \[
    \sup _{x\in {\erre}^d ;\, |x-x_0|<\gamma } 
    \int _{{\erre}^d} q_t(x,y) ^p dy <\infty.
  \]
  Then, $P_t$ has the $L_{p'}$-strong Feller property, where $p'$ be
  the conjugate exponent of $p$.
\end{coroll}
\begin{proof}
  The proof is almost the same as the one of Lemma \ref{thm3-1}.  The
  difference is that we take $f\in L_{p'}$ instead of taking $f \in
  \B_{b,c}(\erre ^d)$, and that we do not need either to take $M>0$
  nor to apply Lemma \ref{lm:Feller}.
\end{proof}

We can now state and prove the main result of this section, which
asserts that if the coefficients of \eqref{eq:fuerte} are sufficiently
smooth and if the diffusion coefficient is uniformly elliptic, we have
the strong Feller property of the associated Markovian semigroup.

Let $X$ be the unique solution to \eqref{eq:fuerte}, where $Y$ is
$d$-dimensional rotationally-invariant $\alpha$-stable process,
$\sigma \in C([0,\infty[ \times \erre^d; \erre^d \otimes \erre^d)$,
$b\in C([0,\infty[ \times \erre^d ; \erre^d)$, and there exists $K>0$
such that
\[
|\sigma (t,x) -\sigma (t,y)|_{\erre^d \otimes \erre^d} 
+ |b(t,x)-b(t,y)|_{\erre^d} \leq K|x-y|, 
\qquad \forall x,y \in {\erre}^d, \; t \in [0,\infty[.
\]
As usual, we shall denote by $P^X_t$ the Markovian semigroup defined by
\[
P_t^X f(x) := \E f(X^x_t), \qquad f \in \B_b(\erre^d).
\]
Note that, thanks to the Lipschitz continuity hypothesis on $b$ and
$\sigma$, it is well known that the solution $X_t^x$ to
\eqref{eq:fuerte} depends continuously on the initial datum $x$, which
in turn implies that $P_t^X$ is Feller.
\begin{thm}\label{thm3-2}
  Assume that there exist positive numbers $\delta$ and $\varepsilon$
  such that $\sigma \in C^{0,2}([0,\delta] \times \erre ^d; \erre^d
  \otimes \erre^d)$, $\nabla \sigma \in C_b^{0,1}([0,\delta] \times
  \erre ^d; \erre^d \otimes \erre^d \otimes \erre^d)$, $b \in
  C^{0,2}([0,\delta ]\times \erre ^d; \erre^d)$, $\nabla b \in
  C_b^{0,1}([0,\delta]\times \erre ^d; \erre^d \otimes \erre^d)$, and
  \[
  |\sigma (t,x)\xi |^2 \geq \varepsilon |\xi |^2, \qquad \forall \xi
  \in \erre^d, \; t \in [0,\delta], \; x \in \erre^d.
  \]
  Then $P_t^X$ is strong Feller for all $t>0$.
\end{thm}
\begin{proof}
  Let $f\in \B _{b}(\erre ^d)$. Let us assume, for the time being,
  that $t\in (0,\delta]$. By \cite[Thm.~6.2]{Sei2}, for each $x\in
  \erre^d$ the density function $q_t(x,\cdot)$ of the distribution of
  $X^x_t$ exists and it belongs to $C_b(\erre^d)$. Furthermore,
  checking the dependence of the estimate for $q_t(x,y)$ in the proof
  \cite[Thm.~6.2]{Sei2}, one infers that $q_t \in \B(\erre^d \times
  \erre^d)$, hence also that \eqref{u-bdd} is satisfied.  We can then
  apply Lemma \ref{thm3-1} obtaining that $P_t^X f \in C_b(\erre ^d)$
  for all $t\in (0,\delta]$. Let now assume $t>\delta$. Since $P_t^X f
  = P_{t-\delta}^X (P_\delta^X f)$ and $P_\delta^X f \in C_b(\erre
  ^d)$, the Feller property of $P_t^X$ for all $t>0$ yields $P_t^X
  f\in C_b(\erre ^d)$.  The theorem is thus proved.
\end{proof}


\bibliographystyle{amsplain}
\bibliography{ref,extra}

\end{document}